\documentclass[10pt,reqno]{amsart}

\usepackage{upgreek} 
\usepackage{lmodern}
\usepackage{anyfontsize}
\usepackage{amssymb, mathtools, thmtools, amscd, xfrac, xspace}
\usepackage[dvipsnames]{xcolor}
\usepackage{amsmath}
\usepackage[alphabetic, msc-links]{amsrefs}
\usepackage{hyperref}
\usepackage{cleveref}
\hypersetup{
	colorlinks=true,
	linkcolor=RubineRed,
	citecolor=MidnightBlue,
	filecolor=Magenta,      
	urlcolor=cyan,
}
\usepackage{enumerate}
\usepackage{bm}
\usepackage{stmaryrd} 
\usepackage[mathcal]{eucal}
\usepackage{verbatim}
\usepackage{array,float}
\usepackage{longtable}
\usepackage{multirow}
\usepackage[margin=1.1in]{geometry}
\usepackage{xy}
\input xy
\xyoption{all}
\usepackage{pdflscape} 
\usepackage{caption} 

\numberwithin{equation}{section}
\numberwithin{equation}{section}
\numberwithin{figure}{section}

\theoremstyle{plain}
\newtheorem{theorem}{Theorem}[section]
\newtheorem{proposition}[theorem]{Proposition}
\newtheorem{lemma}[theorem]{Lemma}
\newtheorem{corollary}[theorem]{Corollary}

\newtheorem*{conjecture*}{Conjecture}

\theoremstyle{definition}
\newtheorem{definition}[theorem]{Definition}

\theoremstyle{remark}

\newcommand{\nonsq}{\epsilon}

\newcommand{\C}{\ensuremath{\mathbb{C}}}

\newcommand{\Z}{\ensuremath{\mathbb{Z}}}

\DeclareMathOperator{\Irr}{\mathrm{Irr}}
\DeclareMathOperator{\irr}{\mathrm{Irr}}

\DeclareMathOperator{\Ind}{Ind}
 
\DeclareMathOperator{\gu}{\mathrm{gu}}

\newcommand{\mat}[4]{\begin{psmallmatrix}
		#1 & #2 \\  #3 & #4 \end{psmallmatrix} }

\newcommand{\smat}[4]{\begin{psmallmatrix}
		#1 & #2 \\  #3 & #4 \end{psmallmatrix} }

\newcommand{\lri}{\mathfrak{o}}
\newcommand{\cO}{\lri}

\renewcommand{\wp}{\mathfrak{p}}

\newcommand{\lup}{{\ell_2}}
\newcommand{\ldown}{{\ell_1}}

\newcommand{\Lri}{\mathfrak{O}}

\newcommand{\mfp}{\wp}
\newcommand{\mfP}{\germ{P}}

\newcommand{\K}{\mathrm{K}}
\newcommand{\G}{\mathrm{G}}
\newcommand{\ind}{\mathrm{Ind}}
\newcommand{\N}{\mathrm{N}}

\newcommand{\co}{\mathfrak{o}}
\newcommand{\g}{\mathfrak{g}}
\newcommand{\gl}{\mathfrak{gl}}
\newcommand{\rl}{R_\ell}
\newcommand{\GL}{\mathrm{GL}}
\newcommand{\gunonsq}{\varepsilon}
 
\newcommand{\GU}{\mathrm{GU}}
\newcommand{\ti}{\tilde}
 \newcommand{\I}{\mathrm{I}}
\newcommand{\nreg}{\mathbf{nreg}}
\newcommand{\cus}{\mathbf{cus}}
\newcommand{\sns}{\mathbf{sns}}
\renewcommand{\ss}{\mathbf{ss}}

\renewcommand{\C}{\mathrm{C}}
\renewcommand{\tt}{\mathfrak{t}}
\newcommand{\tr}{\mathrm{tr}}
\newcommand{\cba}{C_{\breve{A}}}
\newcommand{\cta}{C_{\tilde{A}}}
\renewcommand{\tan}{\mathrm{tangible} }
\begin{document}
\title[Real characters and Real classes of $\GL_2$ and $\GU_2$]{Real characters and real classes of $\GL_2$ and $\GU_2$ over \\ discrete valuation rings}

\author{Archita Gupta}
\address{Department of Mathematics and Statistics, IIT Kanpur,  Kanpur 208016, India}
\email{architagup20@iitk.ac.in}
\author{Tejbir Lohan}
\address{Theoretical Statistics and Mathematics Unit,
Indian Statistical Institute, Delhi Centre, New Delhi 110016, India}
\email{tejbirlohan70@gmail.com}
\author{Pooja Singla}
\address{Department of Mathematics and Statistics, IIT Kanpur,  Kanpur 208016, India}
\email{psingla@iitk.ac.in}



 \begin{abstract}
     Let $\mathfrak{o}$ be the ring of integers of a non-archimedean local field with residue field of odd characteristic, $\mathfrak{p}$ be its maximal ideal and let $\mathfrak{o}_\ell=\mathfrak{o}/\mathfrak{p}^\ell$ for $\ell\ge 2$. In this article, we study real-valued characters and real representations of the finite groups $\GL_2(\mathfrak{o}_\ell)$ and $\GU_2(\mathfrak{o}_\ell)$. We give a complete classification of real and strongly real classes of these groups and characterize the real-valued irreducible complex characters.

We prove that every real-valued irreducible complex character of $\GL_2(\mathfrak{o}_\ell)$ is afforded by a representation over $\mathbb{R}$. In contrast, we show that $\GU_2(\mathfrak{o}_\ell)$ admits real-valued irreducible characters that are not realizable over $\mathbb{R}$. These results extend the parallel known phenomena for the finite groups $\GL_n(\mathbb{F}_q)$ and $\GU_n(\mathbb{F}_q)$. 
 \end{abstract}

\maketitle
\section{Introduction}  
Let $\mathbf F$ be a non-archimedean local field with ring of integers $\mathfrak{o}$, maximal ideal $\mathfrak{p}$, and residue field $\mathsf{k}$ of odd characteristic $p$ and cardinality $q$. Fix a uniformizer $\pi$ of $\mathfrak{o}$. For $\ell \in \mathbb{N}$, let $\mathfrak{o}_\ell=\mathfrak{o}/\mathfrak{p}^\ell$.

Let $\G$ be a classical group scheme defined over $\mathfrak{o}$. The group $\G(\mathfrak{o})$ of $\mathfrak{o}$-points is a maximal compact subgroup of $\G(\mathbf F)$ and plays a fundamental role in the representation theory of $\G(\mathbf F)$. Since $\G(\mathfrak{o})$ is profinite, every continuous finite-dimensional complex irreducible representation of $\G(\mathfrak{o})$ factors through a finite quotient $\G(\mathfrak{o}_\ell)$ for some $\ell$. Consequently, the study of continuous representations of $\G(\mathfrak{o})$ reduces to the study of finite-dimensional representations of the finite groups $\G(\mathfrak{o}_\ell)$.

The representation theory of $\G(\mathfrak{o}_\ell)$ has been studied extensively in recent years; see, for example, \cites{MR4719887, MR4704476, MR4362775} and the references therein. Most of this work focuses on complex representations. In contrast, comparatively little is known about representations over other fields. The purpose of this article is to initiate a systematic study of real representations of $\G(\mathfrak{o}_\ell)$.

We restrict our attention to the groups $\GL_2(\mathfrak{o}_\ell)$ and $\GU_2(\mathfrak{o}_\ell)$. Our first main result is a complete description of the real-valued irreducible complex characters of these groups; see \autoref{thm:regular-self-dual}. An irreducible complex representation of a finite group that admits a realization over $\mathbb{R}$ is called orthogonal (also called real in literature), and its character is called orthogonal. Every orthogonal character is real-valued, but the converse does not hold in general. It is well known that an irreducible character of a finite group is orthogonal if and only if its Frobenius–Schur indicator is equal to $1$; see \cite[Chapter~23]{MR1237401} or \cite[Chapter~4]{MR0460423}. 
In \autoref{thm:orthogonal}, we consider whether every real-valued character of $\G(\cO_\ell)$ is orthogonal. We prove that every real-valued irreducible complex character of $\GL_2(\mathfrak{o}_\ell)$ is orthogonal for $\ell \geq 1$. In contrast, this property fails for $\GU_2(\mathfrak{o}_\ell)$. These results are parallel to the well-known situation for the finite groups $\GL_n(\mathbb{F}_q)$ and $\GU_n(\mathbb{F}_q)$; see \cite[Theorem~4]{MR1656426} and \cite[p.~415]{MR3239291}.

Our approach relies on the known construction of irreducible representations of $\G(\cO_\ell)$ as well on the analysis of real and strongly real  classes. An element $g$ of a group $G$ is called real if it is conjugate to its inverse in $G$, and strongly real if it is conjugate to its inverse by an involution. Equivalently, $g$ is strongly real if and only if it can be written as a product of two involutions. A conjugacy class is called real (respectively, strongly real) if it contains a real (respectively, strongly real) element.

The study of real and strongly real elements has a long history; see \cite{MR3468569} for a comprehensive survey. For a finite group, the number of real  classes equals the number of real-valued irreducible complex characters \cite[Theorem~18.6]{MR1645304}. In the finite field case, Gow \cite{MR649725} classified the real classes of $\GL_n(\mathbb{F}_q)$ and showed that their number coincides with that of $\GU_n(\mathbb{F}_q)$; see also \cite[Theorem~3.8]{MR2794377}. However, unlike $\GL_n(\mathbb{F}_q)$, not every real  class of $\GU_n(\mathbb{F}_q)$ is strongly real. Strongly real classes in $\GU_n(\mathbb{F}_q)$ were enumerated for odd $q$ in \cite{MR3228935}.
We prove analogous results for $\GL_2(\mathfrak{o}_\ell)$ and $\GU_2(\mathfrak{o}_\ell)$ when $\ell \geq 2$, including explicit descriptions and enumerations of real and strongly real classes; see \autoref{thm-real-GL-main}, \autoref{thm-real-GU-main}, \autoref{cor-real-GU-count}, and also \autoref{tab:real_classes}.

The relationship between strongly real elements and orthogonal characters has been studied extensively. For instance,
a conjecture of Tiep, proved by Vinroot~\cite{MR4074058}, states that for finite simple groups every irreducible complex character is orthogonal if and only if every element is strongly real. 
This equivalence fails for finite groups in general; see~\cite{MR3298127}. 
Furthermore, in contrast to $\GU_n(\mathbb{F}_q)$, it is known that for $\GL_n(\mathbb{F}_q)$ every real element is strongly real and every real character is orthogonal; see~\cite{MR3468569,MR1656426,MR3239291}. 
Our results show that the same dichotomy holds for $\GL_2(\mathfrak{o}_\ell)$ and $\GU_2(\mathfrak{o}_\ell)$ when $\ell \geq 2$; see \autoref{sec-orth-char}.

\section{Notation and Preliminaries}
\label{subsec:construction}

In this section, we set up notation and include a few preliminary results.    
Recall that $\lri$ is a complete discrete valuation ring with residue field $\mathsf{k}$ of cardinality $q$ and odd characteristic~$p$.  Let~$\mfp$ be the maximal ideal and let $\pi$ be a fixed uniformizer. Let $\Lri$ be an unramified quadratic extension. It follows that there exists $\gunonsq \in \Lri$ with $\gunonsq^2 \in \lri^\times \smallsetminus (\lri^\times)^2$ such that $\Lri=\lri[\gunonsq]$. Let $\mfP=\pi \Lri$ be the maximal ideal in $\Lri$ and $\mathcal{K}=\Lri/\mfP$ the residue field, a quadratic extension of $\mathsf{k}$ generated by the image of $\gunonsq$. For $\ell \in \mathbb N$, we let $\lri_\ell=\lri/\mfp^\ell$ and  $\Lri_\ell=\Lri/\mfP^\ell$ denote the finite quotients. 
We denote by $x \mapsto x^\circ$ the non-trivial Galois automorphism of $\Lri/\lri$, characterized by $\gunonsq^\circ= -\gunonsq$. The image of $\gunonsq$ in $\Lri_i$ will also be denoted by $\gunonsq$ for all $i$.

Let $W = \smat{0}{1}{1}{0} \in \GL_2(\Lri_\ell)$ denote the permutation matrix corresponding to the  longest Weyl element.  Consider the involution on $\gl_2(\Lri_\ell)$ defined by
\begin{equation}\label{staroperation}
  (a_{i,j})^\star \coloneqq W(a_{j,i}^\circ)W^{-1},
\end{equation}
and its associated Hermitian form on $\Lri_\ell^2$ given by:
\[
  \langle (u_1, u_2),(v_1, v_2) \rangle \coloneqq v_1^\circ u_2 + v_2^\circ u_1.
\]
For $\ell \in \mathbb N \cup \{\infty\}$ the unitary group with respect to $\star$ and its Lie algebra of anti-Hermitian matrices are given by
\[
  \begin{split}
    \GU_2(\lri_\ell) &\coloneqq \left\{ A \in \GL_2(\Lri_\ell) \mid A^\star A=\mathrm{I}_2 \right\},  \\ 
    \gu_2(\lri_\ell) &\coloneqq \left\{ A \in \gl_2(\Lri_\ell) \mid A+ A^\star = \mathrm{O}_2\right\}.
  \end{split}
\]

Throughout this paper we consider  $\GL_2$ and $\GU_2$ as $\lri$-group schemes, where the $R$-points of the latter are the fixed points of $A \mapsto (A^\star)^{-1}$ for every $\lri$-algebra $R$ and $A \in \gl_2(R)$. Let $\g$ be the lie algebra scheme of $\G$. Then $\g$ is either $\gl_2$ or $\gu_2$ as $\lri$-Lie algebra schemes, the latter being the fixed points of $A \mapsto -A^\star$. The adjoint action of a group on its Lie algebra will be denoted by $\mathrm{Ad}$.
Recall $\cO_1 = \mathbb F_q.$ 
\smallskip

 We use $\rl$ as a common notation to denote $\cO_\ell$ for $\G = \GL_2$ and $\Lri_\ell$ for $\G = \GU_2$. For the uniformity in the proofs, we use $\nonsq$ to denote $1$ for $\G=\GL_2$ and $\gunonsq$ for $\G=\GU_2.$ Similarly, $\partial_{\G}$ is used to denote $-1$ for $\G = \GL_2$ and $1$ for $\G = \GU_2$.

For $ i \leq \ell$, let  $\rho_{\ell,i}: \cO_\ell \rightarrow \cO_i$ be the natural projection maps. The corresponding natural projection maps $\G(\cO_{\ell})  \rightarrow \G(\cO_i) $ are also denoted by $\rho_{\ell, i}$. For any matrix $A \in \g(\co_\ell)$, we denote $\rho_{\ell,1}(A)$ by $\rho_{\ell,1}({A}).$ Let $\K^{i} =  \ker (\rho_{\ell,i}) $ be the $i$-th congruence subgroups of $\G(\cO_\ell).$  For $i \geq \ell/2,$  the group  $\K^i $ is isomorphic to the abelian additive subgroup $\g(\cO_{\ell-i})$ of $M_2(R_{\ell-i}).$ Let $\psi: \rl  \rightarrow \mathbb C^\times$ be a fixed primitive one dimensional representation of $\rl$. For $ \rl=\Lri_\ell,$ we assume that $\psi$ satisfies $\psi(x+\nonsq y)= \psi'(x)\psi'(y)$ for some primitive one dimensional representation $\psi'$ of $\cO_\ell$. Therefore,  $\pi^{\ell-1}\cO_\ell \not\subseteq \ker(\psi) $ by our choice of $\psi$.

 For any $i \leq \ell/2$ and $A = [a_{st}] \in \g(\cO_i),$ we will consider lifts $\tilde{A} = [\widetilde{a_{st}}] \in \g(\cO_\ell)$ of $A$ such that $\rho_{\ell,i}(\tilde{A}) = A$ with   $\widetilde{a_{st} } = \nonsq$ for $a_{st} = \nonsq,$  $\widetilde{a_{st}} = 0$ for $a_{st} = 0$ and $\widetilde{a_{ij}} = \widetilde{a_{kl}}$ whenever $a_{ij} = a_{kl}$. In this case, we say $\tilde{A}$ is a {\bf Serre lift} of $A$.  
 For any $i \leq \ell/2$ and  $A \in \g(\cO_i)$, define $\psi_A: \I+ \pi^{\ell-i} \g(\cO_\ell) \rightarrow \mathbb C^\times$ by 
 \begin{equation}
 \label{eqn:psi-A definition}
      \psi_A(\I+ \pi^{\ell-i} B) \coloneqq \psi(\pi^{\ell-i}\bm{tr}(\tilde{A}B)),
 \end{equation}
for any lift $\ti{A}$ of $A$ and $\I+ \pi^{\ell-i} B \in \K^{\ell-i}.$  Then $\psi_A$ is a well defined one dimensional representation of $\K^{\ell-i}.$ Further, the following duality for abelian groups $\K^{i}$ and $\g(\cO_{\ell-i})$ holds for $i \geq \ell/2$.
\begin{equation}
\label{eq:duality}
\g(\cO_{\ell-i}) \cong \widehat {\K^{i}}\,\,; A \mapsto \psi_A\,\, \mathrm{where}, \,\, \psi_A(\I+ \pi^{i} B) = \psi(\pi^{i}\bm{tr}(\tilde{A}B))  \,\, \forall \,\, \I+ \pi^{i} B \in \K^{i}. 
\end{equation}

We use \cite{MR1334228} and \cite{MR3737836} to recall the notion of regular elements and regular characters of $\G(\cO_\ell)$, see also \cite[Lemma~2.3]{MR4399251}. An element $A \in \g(\cO_\ell)$ is called regular  if and only if $\rho_{\ell,1}({A}) \in \g(\cO_{1})$ is a regular matrix (that is the characteristic polynomial of $\rho_{\ell,1}({A})$ is equal to its minimal polynomial).  The stabilizer of the regular $A$ in $\G(\cO_{\ell-i})$ under the conjugation action is $\{x \I + yA \mid x, y \in R_{\ell-i}\} \cap \G(\cO_{\ell-i}).$
 For $i \geq \ell/2$, the representation $\psi_A \in \widehat{\K^{i}}$ is called regular if and only if $\psi_A|_{ \K^{\ell-1}}$ is regular and an irreducible representation $\rho$ of $\G(\cO_\ell)$ is called regular if the $\mathrm{Ad}$-orbit of its restriction to $\K^{\ell-1} $ consists of one dimensional regular representations.  By \cite[Lemma~2.3]{MR4399251}, $\psi_A \in \widehat{\K^{i}}$ is regular if and only if $A$ is regular.  

The following lemma describes the orbits of $\g(\cO_\ell)$ under the  $\mathrm{Ad}$-action of $\G(\cO_\ell)$. 
\begin{lemma}
\label{lem:orbit-representatives-gol}
An exhaustive list of $\g(\cO_\ell)$ orbit representatives under the $\mathrm{Ad}$-action of $\G(\cO_\ell)$ is given by matrices $A \in \g(\cO_\ell)$ of the following form:
\begin{enumerate}
    \item[(a)] $x\I+ \pi C $ 
    \item[(b)] $\smat{x}{\nonsq \pi \beta}{  \nonsq}{x} $
    \item[(c)] $\smat{x-r \nonsq^2}{0}{  0}{x+r \nonsq^2} $ with $r \in \cO_\ell^\times.$  
        \item[(d)] $\smat{x}{\nonsq \sigma }{\nonsq  }{x} $ with $\sigma \in (\cO_\ell^\times)^2$ for $\g = \gu_2$ and   $\sigma \in 
\cO_\ell^\times \setminus (\cO_\ell^\times)^2$ for $\g = \gl_2.$  
\end{enumerate}
\end{lemma} 
\begin{proof}
    For $\GL_2$, proof follows from \cite[Section~2]{MR2584957}. For $\GU_2,$ we note that $A  \in \gu_2(\cO_\ell)$ if and only $A$ is anti-hermitian. If $A$ is a scalar modulo $\pi,$ then $A$ is of type (a). Otherwise the result follows from Lemma~\cite[Lemma~3.5]{MR3471251}. 
\end{proof}

 Define $\tt: \g(\lri_\ell) \rightarrow \{\nreg ,  \sns, \ss, \cus  \}   $ by 
$ \tt(A) = \nreg\,(\sns, \ss, \cus)$ if $A$ is equivalent to a matrix given in above (a) ((b), (c), (d)).  We note that these correspond to $\rho_{\ell,1}({A})$ being non-regular, split non-semisimple, semisimple and cuspidal, respectively.  Notice that $A\in \g(\co_{\ell}) $ can either be regular or scalar. Recall that a representation $\rho$ of $\G(\co_\ell)$ is called a \emph{twist} of $\rho'$, if $\rho \cong \chi \otimes \rho'$ for a one dimensional representation $\chi$ of $\G(\co_\ell)$. For $\rho_{\ell,1}({A})$ scalar, the irreducible representations of $\G(\co_\ell)$ lying above $\psi_A$ are called {\it non-regular} representations of $\G(\co_\ell)$. Equivalently, an irreducible representation $\rho$ of $\G(\co_\ell)$ is called \emph{non-regular} if either $\rho$ or any of its twists can be obtained from $\G(\co_{\ell-1})$. 

For $\ell \geq 2$, we use $\ldown$ and $\lup$ to denote $\lfloor \frac{\ell}{2} \rfloor$ and $\lceil\frac{\ell}{2}  \rceil$, respectively. A regular representation $\rho$ of $\G(\cO_\ell)$ is called $\ss (\sns, \cus)$ if the $\mathrm{Ad}$-orbit of its restriction to $\K^{\ell-1} $ consists of one dimensional representations $\psi_A \in \widehat{\K^\lup}$ for $A \in \g(\cO_\ldown)$ with $\tt(A) = \ss (\sns, \cus)$.

 Now we summarize very briefly the construction of regular representations of $\G(\cO_\ell) $ with emphasis on the statements that we require in this article.  
\subsection{Construction of regular representations of $\G(\cO_\ell)$ for $\ell$ even}
\label{E.construction}
Let $\psi_A \in \widehat{\K^{\ell/2}}$ be a regular one dimensional representation of $\K^{{\ell/2}}$ for $A \in \g(\cO_{\ell/2}).$ Then the following gives the construction in this case. Let   $ S_A = \{ g \in \G(\cO_\ell) \mid \psi_A^g \cong \psi_A  \}  $ be the inertia group of $\psi_A$ in $\G(\cO_\ell).$ 
Let $\tilde{A} \in \g(\cO_\ell)$ be a lift of $A$, and  let  $\C_{\G(\lri_\ell)}(\tilde{A})$
denote its  stabilizer  in $\G(\co_\ell)$ under the $\mathrm{Ad}$-action. Then $ S_A  =  \C_{\G(\lri_\ell)}(\tilde{A})       \K^{{\ell/2}} .$  
Let $\rho \in \mathrm{Irr} \left( \G(\cO_\ell)  \mid \psi_A    \right)$ be a regular representation of $\G(\cO_\ell),$ then there exists an extension $\widetilde{\psi_A}$ of $\psi_A$ to $S_A$ such that $\rho \cong \mathrm{Ind}_{S_A}^{\G(\cO_\ell)} (\widetilde{\psi_A}) .$
  Every $\rho \in \mathrm{Irr} \left( \G(\cO_\ell) \mid \psi_A \right)$ has dimension $\frac{|\G(\cO_\ell)|}{|\mathrm{C}_{\G(\cO_{\ell/2})}(A)| |\K^{\ell/2}|} .$
\subsection{Construction of regular representations of $\G(\cO_\ell)$ for $\ell $ odd}
 We briefly explain the construction of irreducible representations of $\G(\co_\ell)$ for odd $\ell$ analogous to the one given in \cite{MR2588859} for $\G=\GL_2$. This construction also appears in \cite[Section~3]{MR2584957} for $\GL_2(\mathbb{Z}/p^\ell\mathbb{Z})$ and in \cite[Section~4.H.2]{Campbell-thesis} for $\G=\GU_2$.
Every regular character of $\K^{\lup}$ is of the form $\psi_A$. The inertia group of $\psi_A$ in $\G(\cO_\ell)$ satisfies $S_A = \C_{\G(\cO_\ell)}(\ti{A}) \K^\ldown.$ By Clifford theory, to describe $\mathrm{Irr}(\G(\co_\ell) \mid \psi_A),$ it is enough to describe $\mathrm{Irr}(S_A \mid \psi_A).$ 

\subsubsection{A construction for split non-semisimple and split semisimple representations of $\G(\co_\ell)$ for odd $\ell$ }
\label{subsec:alternate const for ss,sns}


Recall from \autoref{lem:orbit-representatives-gol}, any $A \in \g(\cO_\ldown)$ such       $\tt(A) = \sns$, up to conjugation, has form $\smat{\alpha}{\nonsq \pi \beta}{\nonsq}{\alpha}$ and similarly any $A \in \g(\cO_\ldown)$ such that $\tt(A) = \ss$ has the form $\smat{\alpha}{0}{0}{\beta}$. Then, it is easy to see that  $\mathrm{N}=\left\{\smat{ 1+\pi^{\ldown}x}{\pi^{\lup}z}{ \pi^{\ldown}y}{ 1+\pi^{\ldown}w} \mid x,y,z,w \in \rl \right \} \cap \G(\co_\ell)$ is a normal subgroup of $S_A$.
Consider the lifts 
$\ti{A}=\smat{  \ti{\alpha}}{\nonsq \pi \ti{\beta}}{ \nonsq}{\ti{\alpha}}\in \g(\co_\ell)$ of $A$ such that $\tt(A) = \sns$ and lifts 
$\ti{A}=\smat{\ti{\alpha}}{0}{0}{\ti{\beta}}$for $A$ with $\tt(A) = \ss$.

 \noindent  Define an extension 
   $\psi_{\ti{A}}$ of $\psi_A$ to $\N$ as follows:
\begin{equation}
\label{eqn: definition of psi A on N}
    \psi_{\ti{A}}\smat{ 1+\pi^{\ldown}x}{\pi^{\lup}z}{ \pi^{\ldown}y}{ 1+\pi^{\ldown}w}\coloneqq\psi\left(\pi^{\ldown}\bm{tr}\left(\ti{A}\smat{x}{\pi z}{y}{w}-\frac{\pi^{\ldown}}{2}\ti{A}\smat{x}{\pi z}{y}{w}^2\right)\right). 
\end{equation} 
The stabilizer of $\psi_{\ti{A}}$ in $ S_A$ is given by $ \mathrm{N}  \mathrm{C}_{\G(\co_\ell)}(\ti{A})$. For the ease of the notation, we use $\cta$ to denote $\mathrm{C}_{\G(\co_\ell)}(\ti{A})$.  Since $\cta$ is abelian, the character $\psi_{\ti{A}}$ extends to a character $\psi'_{\ti{A}}$ of $\mathrm{N} \mathrm{C}_{\G(\co_\ell)}(\ti{A})$ and every character of $\mathrm{N} \mathrm{C}_{\G(\co_\ell)}(\ti{A})$ lying above $\psi_{\ti{A}}$ is one dimensional. Using Clifford theory for the group $S_A$ and its normal subgroup $\N$ and its characters of the form $\psi_{\ti{A}}$, we obtain $\ind_{\mathrm{N} \cta}^{S_A} \psi_{\ti{A}}'$ is an irreducible representation of $S_A$ of dimension $q$. Further, every representation $\phi \in \mathrm{Irr}(S_A \mid \psi_A)$ for $\tt(A) \in \{\sns, \ss\}$ satisfies  
\begin{equation} 
\label{eqn:SA-representation-sns-ss}
\phi \cong \ind_{\mathrm{N} \cta}^{S_A} \psi_{\ti{A}}'
\end{equation} 
for some $\ti{A}$ and some $\psi_{\ti{A}}'.$ Moreover, if $\tilde A_1$ and $\tilde A_2$ are two lifts of $A$ giving rise to distinct characters $\psi_{\tilde A_1}$ and $\psi_{\tilde A_2}$ of $\mathrm{N}$, then the induced representations $\Ind_{\mathrm{N}\cta}^{\G(\co_\ell)} \psi_{\tilde A_1}'$ and $
\Ind_{\mathrm{N}\cta}^{\G(\co_\ell)} \psi_{\tilde A_2}'$ are non-isomorphic. We will require the following result regarding the groups $\N \cta$ and characters $\psi_{\ti{A}}'.$  
\begin{lemma}
\label{lem:info about ss and sns odd case new lemma}
For $\tt(A) = \ss$ take $\ti{A} = \smat{ \ti{\alpha}}{0}{0}{- \ti{\alpha}}\in \g(\co_\ell)$ and for $\tt(A) = \sns$ take $\ti{A} = \smat{0}{\nonsq \pi \ti{\beta}}{ \nonsq }{0}\in \g(\co_\ell)$. Then the equation
\[
\underset{h\in \N\cta \backslash \G(\co_\ell)/ (\N\cta)^g}{\oplus} {\psi_{\ti{A}}'} \otimes ({\psi_{\ti{A}}'})^{hg}=\mathbf{1} \text{ on } {\N\cta\cap (\N\cta)^{gh}}
\]
is equivalent to  $ {\psi_{\ti{A}}'}\otimes {\psi_{\ti{A}}'}^g=\mathbf{1}  \text{ on }
\cta$, where $g = \smat{0}{1}{1}{0}$ for $\tt(A)=\ss$ and $g=\smat{z}{0}{0}{-z} \in \G(\cO_\ell)$ for $\tt(A)=\sns$. 

\end{lemma}
\begin{proof}
For proving this lemma we consider $\tt(A)=\ss$ and  $\tt(A)=\sns$ separately.
First we consider the split semisimple case. In this case, $\ti{A} = \smat{ \ti{\alpha}}{0}{ 0}{- \ti{\alpha}}$ and $g = \smat{0}{1}{1}{0}$. First we claim that $\N\cta \backslash S_A / (\N \cta)^g=\I.$  For $\smat{a}{\pi^{\ldown}b}{\pi^{\ldown}c}{d}\in S_A$, take $X=\smat{a}{0}{\pi^{\ldown}c }{(d-\pi^{2\ldown}cba^{-1})}\in \N \cta$ and $Y=\smat{1}{\pi^{\ldown}a^{-1}b }{0}{1}\in (\N \cta)^g.$ Then by direct computations we get $XY=\smat{a}{\pi^{\ldown}b}{\pi^{\ldown}c}{d}$ which proves our first claim. Next we claim that $\psi'_{\ti{A}}\otimes (\psi'_{\ti{A}})^g=\bf{1}$ on $\N \cap \N^g$. Notice that $\psi'_{\ti{A}}=\psi_{\ti{A}}$ on $\N$. Take $\smat{ 1+\pi^{\ldown}a}{\pi^{\lup}b}{ \pi^{\lup}c}{ 1+\pi^{\ldown}d}\in \N \cap \N^g.$ By direct calculations using \autoref{eqn: definition of psi A on N} we get $(\psi_{\ti{A}}\otimes (\psi_{\ti{A}})^g) (X)=1.$ These two claims and the fact that $\N\cta \cap \N \cta^g=(\N\cap \N^g)\cta,$ prove the lemma for this case.

Now we consider the split non-semisimple case. In this case, $\ti{A} = \smat{0}{\nonsq \pi \ti{\beta}}{ \nonsq }{0}\in \g(\co_\ell)$ and  $g=\smat{z}{0}{0}{-z} \in \G(\cO_\ell).$ First notice that $(\N \cta)^g=\N \cta.$ The set of double coset representatives of $\N \cba \backslash S_A / \N \cba$ is given by $$\Gamma=\{\smat{1}{0}{0}{1}, \smat{1}{\nonsq\pi^{\ldown}w}{0}{1}, w \in \co_1^\times\} .$$ 
The given equation becomes
\begin{equation}
\label{eqn: sns eqn on NC A}
    \underset{h\in\Gamma}{\oplus} {\psi_{\ti{A}}'} \otimes ({\psi_{\ti{A}}'})^{hg}=\mathbf{1} \text{ on } {\N\cta}
\end{equation}
Now we claim that ${\psi_{\ti{A}}'} \otimes ({\psi_{\ti{A}}'})^{hg}=\mathbf{1} \text{ on } {\N}$ if and only if $h=\smat{1}{0}{0}{1}.$ Observe that if this claim holds then using \autoref{eqn: sns eqn on NC A}, the lemma follows. We first note that the claim is true for $h=\smat{1}{0}{0}{1}.$ Now we proceed to prove the other side. Take $h=\smat{1}{\nonsq\pi^{\ldown}w}{0}{1}\in \Gamma.$ By using $\psi'_{\ti{A}}=\psi_{\ti{A}}$ on $\N$ and \autoref{eqn: definition of psi A on N} we get
\[
{\psi_{\ti{A}}'} \otimes ({\psi_{\ti{A}}'})^{hg}\smat{ 1+\pi^{\ldown}a}{\pi^{\lup}b}{ \pi^{\ldown}c}{ 1+\pi^{\ldown}d}=\psi(\pi^{2\ldown}\nonsq^2 w(d-a))
\]
For $\G=\GL_2,$ using $\pi^{\ell-1}\cO_\ell \not\subseteq \ker(\psi) $ we get that $\psi(\pi^{2\ldown}\nonsq^2 w(d-a))= 1$ for all $\smat{ 1+\pi^{\ldown}a}{\pi^{\lup}b}{ \pi^{\ldown}c}{ 1+\pi^{\ldown}d}\in \N$ if and only if $w=0\mod (\pi).$ For $\G=\GU_2,$ by the defining conditions of $\GU_2(\co_\ell)$, $\smat{ 1+\pi^{\ldown}a}{\pi^{\lup}b}{ \pi^{\ldown}c}{ 1+\pi^{\ldown}d}\in \N$ implies $(1+\pi^{\ldown}a)(1+\pi^{\ldown}d^\circ)=1.$ Using this we get $\pi^{2\ldown}(d-a)=\pi^{2\ldown}(d+d^\circ).$ Substituting this in $\psi(\pi^{2\ldown}\nonsq^2 w(d-a))$ we get $\psi(\pi^{2\ldown}\nonsq^2 w (d+d^\circ)).$ Since $\nonsq^2 w (d+d^\circ)\in \co_\ell,$ using $\pi^{\ell-1}\cO_\ell \not\subseteq \ker(\psi) $, we get $w=0 \mod (\pi).$ This proves the claim. 
\end{proof}

\subsubsection{ Cuspidal representations of $\G(\co_\ell)$ for odd $\ell$:}
\label{subsec:alternate const for cus} 
In this section, we include a few results from the construction of cuspidal representations of $\G(\cO_\ell)$ that we require for our results. 
Let $A=\smat{0}{\nonsq \alpha}{\nonsq}{0}\in \g(\co_{\ldown})$ be a regular matrix with $\tt(A)=\cus$.  Define  $ D^{\ell_i}(\ti{A})\coloneqq(\C_{\G(\cO_\ell)}(\ti{A})\cap \K^1)\K^{\ell_i} \text{ for } i \in \{1,2\}$. Let $\Z_\ell$ denote the center of $\G(\co_\ell).$
The character $\psi_A$ can be extended to $\Z_\ell D^{\ell_2}(\ti{A})$, say $\widetilde{\psi_A}$. We have $\Z_\ell D^{\ell_2}(\ti{A}) \trianglelefteq \Z_\ell  D^{\ell_1}(\ti{A})$ and every element of $\Z_\ell  D^{\ell_1}(\ti{A})$ stabilizes $\widetilde{\psi_A}$.
By considering the Heisenberg lifts, there exists a unique irreducible representation of $\Z_\ell  D^{\ldown}(\ti{A})$ which extends to $S_A.$ In this way, we obtain a construction of every $\phi \in \mathrm{Irr}(S_A \mid \psi_A)$, see \cite[Section~4.5]{gupta2025tensorproductsregularcharacters} for details. The following results holds for the character $\chi_\phi$ of  $\phi$. 
\begin{proposition}\label{prop:character values}
Let $\ell$ be odd and $A=\smat{0}{\nonsq \alpha}{\nonsq}{0}\in \g(\cO_{\ell_1})$ be regular such that $\tt(A)=\cus.$
    For  $\phi  \in \mathrm{Irr}(S_{A}\mid \psi_A),$ the character $\chi_\phi$ of $\phi$ satisfies the following:
    \begin{enumerate}
        \item $\chi_\phi (g)=q\widetilde{\psi_A}(g)$ for all $g\in \Z_\ell  D^{\ell_2}(\tilde{A}),$ where $\widetilde{\psi_A}\in \mathrm{Irr}(\Z_\ell  D^{\ell_2}(\tilde{A})\mid \psi_A)$ such that \\ $\langle \mathrm{Res}^{S_A}_{\Z_\ell   D^{\ell_2}(\tilde{A})}(\phi),\widetilde{\psi_A}\rangle\neq 0.$
        \item $\chi_\phi (g)=0$ for all $g\in \Z_\ell  D^{\ell_1}(\tilde{A})\setminus \Z_\ell  D^{\ell_2}(\tilde{A}).$
        \item $ |\chi_\phi(g)|=1$  for all $ g \in S_A\setminus\Z_\ell  D^{\ell_1}(\tilde{A}).$
    \end{enumerate}
\end{proposition}
\begin{proof}
  For a proof of this, see \cite[Proposition~5.2]{gupta2025tensorproductsregularcharacters}.
\end{proof}

As mentioned earlier, for $m\in \mathbb{N},$ the stabilizer of the regular $A$ in $\g(\cO_{m})$ under the conjugation action is $\{x \I + yA \mid x, y \in R_{m}\} \cap \G(\cO_{m}).$ For $\G=\GL_2,$ by using this fact it is easy to calculate the stabilizer cardinality $|\C_{\GL_2(\co_\ell)}({A})|.$ For $\G=\GU_2,$ the cardinality of these stabilizers is calculated in \cite[Section~4.H.1,4.H.2]{Campbell-thesis}.

\begin{lemma}
\label{lem:cardinalities of centralizers}
 For $A\in \g(\co_\ell),$ we have   $|\C_{\G(\co_\ell)}({A})|=\begin{cases}
 (q+1)(q- \partial_{\G})q^{2\ell-2}; & \tt(A)=\cus,\\
    (q-\partial_{\G}) q^{2\ell-1};  & \tt(A)=\sns,\\
      (q-1)(q-\partial_{\G})q^{2\ell-2};   & \tt(A)=\ss.
    \end{cases}$
\end{lemma}

\section{Real and strongly real  classes of $\G(\mathfrak{o}_{\ell})$} 
In this section, we describe and enumerate the real and strongly real classes in $\G(\mathfrak{o}_{\ell})$. We also show that every real element of $\mathrm{GL}_{2}(\mathfrak{o}_{\ell})$ is strongly real. 

\subsection{Real 
classes of $\mathrm{GL}_{2}(\mathfrak{o}_{\ell})$}
The following result regarding the conjugacy classes of $\GL_2(\cO_\ell)$ was proved in \cite[Theorem 2.2]{MR2543507}. 
\begin{theorem}
\label{thm:conjugacy classes GL2}
Any $A \in M_2(\mathfrak{o}_\ell)$ is similar to a unique matrix of the form 

\begin{eqnarray}
\label{eq:definition-of-M}
M(d, i, \alpha, \beta) &= & dI + \pi^i \mat 0{\alpha}1{\beta},
\end{eqnarray}
where $i\in \{0,1, \ldots, \ell\},$ $d = \sum_{k=0}^{i-1} d_k \pi^k$ with $d_i \in \cO_1,$ and $\alpha, \beta$ are determined modulo $(\pi^{\ell - i}).$ 
 Further $M(d, i, \alpha, \beta) \in \GL_2(\cO_\ell)$ if and only if $d_0 \in \mathfrak{o}_{1}^{\times}$ for $i \neq 0$ and $\alpha \in \cO_{\ell}^\times$ for $i=0$.
\end{theorem}

We use this result to describe the real and strongly real elements of $\GL_2(\cO_\ell).$ By definition, $A \in \GL_2(\cO_\ell)$ is real if and only $A$ is conjugate to $A^{-1}.$ Therefore $A$ is real implies that $\det(A) = \det(A^{-1})$ and $\tr(A) = \tr(A^{-1}).$ In our next result, we prove that the converse also holds for $\GL_2(\cO_\ell)$. 

\begin{theorem}\label{thm-real-GL-main}
Let $A\in\mathrm{GL}_{2}(\mathfrak{o}_{\ell})$. The following statements are equivalent:
\begin{enumerate}
    \item $A$ is real.
    \item $A$ is strongly real.
    \item $\det(A)\in\{\pm1\}$ and $\tr(A)=\tr(A^{-1})$.
    \item $A$ is conjugate to one of the following matrices:
    \begin{enumerate}
        \item $M(0,0,1,0)$;
        \item $M(0,0,-1,\beta)$ with $\beta\in\cO_\ell$;
        \item $M(1,i,\alpha,\pi^i\alpha)$ for $1\le i\le\ell$ and $\alpha\in\cO_\ell$;
        \item $M(-1,i,\alpha,-\pi^i\alpha)$ for $1\le i\le\ell$ and $\alpha\in\cO_\ell$.
    \end{enumerate}
\end{enumerate}
\end{theorem}

\begin{proof}
The implication (1)\,$\implies$\,(3) is immediate, and (2)\,$\implies$\,(1) follows directly from the definitions.
The equivalence (3)\,$\iff$\,(4) follows from the parametrization in \autoref{eq:definition-of-M}, since the determinant and trace conditions precisely describe the families listed in (4). Thus, it remains to show that every matrix appearing in (4) is strongly real. 

\noindent For $M(d,i,\alpha,\beta)$ with $i=0$ among the matrices in (4), the involution $\mat{0}{1}{1}{0}$ satisfies
\[
\mat{0}{1}{1}{0}\,M(0,0,\alpha,\beta)\,\mat{0}{1}{1}{0}^{-1}
= M(0,0,\alpha,\beta)^{-1}.
\]
For $M(d,i,\alpha,\beta)$ with $i>0$ among the matrices in (4), the involution $\mat{1}{\beta}{0}{-1}$ satisfies
\[
\mat{1}{\beta}{0}{-1}\,M(d,i,\alpha,\beta)\,\mat{1}{\beta}{0}{-1}^{-1}
= M(d,i,\alpha,\beta)^{-1}.
\]
Hence, every matrix in (4) is strongly real, completing the proof.
\end{proof}

Observe that $M(d,i,\alpha,\beta)$ is regular (non-regular) if and only if $i=0 \,(1\leq i\leq \ell)$. As a direct consequence of \autoref{thm-real-GL-main}, we obtain the following corollary.

\begin{corollary}\label{cor-real-GL-count}
\begin{enumerate}
    \item Total number of  real regular and real non-regular classes of $\GL_2(\co_\ell)$ is $q^\ell+1$ and $2 \sum_{0}^{\ell-1} q^{i}.$
    \item The total number of real (strongly real) classes of $\mathrm{GL}_{2}(\mathfrak{o}_{\ell})$ is $ 1 + q^{\ell} + 2 \sum_{0}^{\ell-1} q^{i}.$
\end{enumerate}
\end{corollary}

\subsection{Real  classes in $\GU_2(\cO_\ell)$} 
The following result regarding the conjugacy  classes of $\GU_2(\cO_\ell)$ follows from \cite[Section 4.K]{Campbell-thesis}.
\begin{theorem}[{\cite[p. 69]{Campbell-thesis}}]\label{lem-conjugacy-unitary}
Every element of $\GU_2(\cO_\ell)$ is conjugate to one of the following matrices:
\begin{enumerate}[(A)]
\item $\smat{x}{0}{0}{x}$, where $x x^\circ=1$.
\item $\smat{x}{0}{0}{{{x}^{-1}}^\circ}$, where $x {x}^\circ \neq 1$.
\item $\smat{x}{y}{y}{x}$, where $y \neq 0$ such that
$xx^\circ +yy^\circ =1 \text{ and } xy^\circ +x^\circ y=0.$
\item $\smat{x}{\pi^{i+1} \beta y}{\pi^{i} y}{x}$, where $0\leq i \leq \ell-1$, $\beta \in \cO_{\ell-i-1}$ and $x,y$ are units in $\Lri_\ell$ such that
$ xx^\circ  + \pi^{2i+1} \beta y y^\circ =1 \text{ and }   \pi^{i} (x y^\circ +x^\circ y)=0.$
\end{enumerate}
\end{theorem}

\noindent Our next result characterizes the real elements of $\GU_2(\cO_\ell).$
\begin{theorem}\label{thm-real-GU-main}

Let $A \in \mathrm{GU}_{2}(\mathfrak{o}_{\ell})$. The following statements are equivalent:
\begin{enumerate}
\item $A$ is real.
\item $\det(A) \in \{ \pm 1\}$ and $\tr(A) = \tr(A^{-1}).$
\item Up to conjugacy, $A$ has one of the following form:
\begin{enumerate}
\item $\smat{x}{0}{0}{x}$, where $x \in \{\pm 1 \}$.
\item $\smat{x}{0}{0}{x^{-1}}$, where $x \in \co_\ell^{\times}$ such that $x^2 \neq 1$.
\item $\smat{0}{1}{1}{0}$ or $\smat{x}{y}{y}{x}$,
where $y \neq 0$ such that
$xx^\circ+yy^\circ=1$, $xy^\circ+x^\circ y=0,$ and $x^2 -y^2 =1$.
\item $\smat{x}{\pi^{i+1} \beta y}{\pi^{i} y}{x}$, where $0\leq i \leq \ell-1$, $\beta \in \cO_{\ell-i-1}$, $x \in \cO_{\ell}^{\times}$, and $y \in \Lri_\ell^{\times}$ such that
$ x^2 + \pi^{2i+1} \beta y y^\circ=1  = x^2 - \pi^{2i+1} \beta y^2 \text{ and } \pi^{i} (x y^\circ +x^\circ y)=0.$
\end{enumerate}
\end{enumerate}
\end{theorem}

\begin{proof} Here, (1) $\implies$ (2) follows from the definition of real elements. We now prove (2) $\implies$ (1). For this, we identify the conjugacy classes satisfying (2) from \autoref{lem-conjugacy-unitary} and show that each also satisfies (1).

\begin{enumerate}
\item[(A)] Consider $A = \smat{x}{0}{0}{x}$ with $x x^\circ = 1$ satisfying (2). Then $\operatorname{tr}(A) = \operatorname{tr}(A^{-1})$ implies $x = \pm 1$. Thus $A$ satisfies (1).

\item[(B)] Consider $A = \smat{x}{0}{0}{{x^{-1}}^\circ}$ with $x {x}^\circ \neq 1$ satisfying (2). Then $\det(A) = x {x^{-1}}^\circ\in \{\pm 1\}$, so $x = x^\circ$ or $x = -x^\circ$. From $\operatorname{tr}(A) = \operatorname{tr}(A^{-1})$, we obtain
\[
(x x^\circ  + 1)(x - x^\circ ) = 0.
\]
The case $x = -x^\circ $ yields $x^2 = 1$, contradicting $x x^\circ  \neq 1$. Hence $x = x^\circ $, so
\[
A = \smat{x}{0}{0}{x^{-1}}, \quad x \in \co  _\ell^\times, \ x^2 \neq 1.
\]
The involution $\tau = \smat{0}{1}{1}{0} \in \mathrm{GU}_2(\mathfrak{o}_\ell)$ satisfies $\tau A \tau^{-1} = A^{-1}$. Thus $A$ satisfies (1).

\item[(C)] Consider $A = \smat{x}{y}{y}{x}$ with $y \neq 0$, $x x^\circ  + y y^\circ  = 1$, and $x y^\circ  + x^\circ  y = 0$, satisfying (2). Then $\operatorname{tr}(A) = \operatorname{tr}(A^{-1})$ yields two cases:

\noindent {\bf Case 1: $\det(A) = -1$.} Here $x = 0$, so $A = \smat{0}{y}{y}{0}$ with $y \in \{\pm 1\}$. Thus $A$ is conjugate in $\mathrm{GU}_2(\mathfrak{o}_\ell)$ to the involution $\smat{0}{1}{1}{0}$, hence satisfies (1).

\noindent {\bf Case 2: $\det(A) = 1$.} Then $A^{-1} = \smat{x}{-y}{-y}{x} $. For any $u + v \nonsq \in \Lri_\ell$ with $u,v\in \co_\ell$ such that $(u + v \nonsq) {(u + v \nonsq)}^\circ = -1$, we have
\[
\smat{v \nonsq}{u}{-u}{-v \nonsq }
A \smat{v \nonsq}{u}{-u}{-v \nonsq }^{-1} = A^{-1}.
\]
Such elements exist by surjectivity of the norm map. Thus $A$ satisfies (1).

\item[(D)] Consider $A = \smat{x}{\pi^{i+1} \beta y }{\pi^i y}{x}$ with $x, y$ units in $\Lri_\ell$ satisfying (2). Then $\operatorname{tr}(A) = \operatorname{tr}(A^{-1})$ implies $\det(A) = 1$. As in Case C, $\smat{\omega}{0}{0}{-\omega} A \smat{\omega}{0}{0}{-\omega}^{-1} = A^{-1}$ for any $\omega \in \Lri_\ell^\times$ with $\omega {\omega}^\circ = -1$. Thus $A$ satisfies (1).
\end{enumerate}

Therefore, (2) $\implies$ (1). The above analysis of cases (A)--(D) also establishes the equivalence of (2) and (3). This completes the proof.
\end{proof}

Unlike $\mathrm{GL}_{2}(\mathfrak{o}_{\ell})$, not every real element in $\mathrm{GU}_{2}(\mathfrak{o}_{\ell})$ is strongly real. The following result characterizes strongly real elements in $\mathrm{GU}_{2}(\mathfrak{o}_{\ell})$ and demonstrates that case (d) of \autoref{thm-real-GU-main} contains no strongly real elements.

\begin{theorem}\label{cor-str-real-GU}
A real element in $\mathrm{GU}_{2}(\mathfrak{o}_{\ell})$ is strongly real if and only if it is conjugate to one of the elements described in parts (a)--(c) of \autoref{thm-real-GU-main}. 
\end{theorem}

\begin{proof}
Since reality is preserved under conjugacy, for the sufficiency part we consider the involutions constructed in the proof of \autoref{thm-real-GU-main}, which conjugate the elements described in parts~(a)--(c) of \autoref{thm-real-GU-main} to their inverses.

For the necessity part, it suffices to show that if $A$ is a real element in $\mathrm{GU}_{2}(\mathfrak{o}_{\ell})$ conjugate to the matrices described in part (d) of \autoref{thm-real-GU-main}, then $A$ is not strongly real. We proceed as follows.

Consider
$
A = \smat{x}{\pi^{i+1} \beta y}{\pi^{i} y}{x},
$
where $0 \leq i \leq \ell-1$, $\beta \in \cO_{\ell-i-1}$, $x \in \Lri_{\ell}^{\times}$, and $y \in \Lri_\ell^{\times}$ such that
$$
x^2 + \pi^{2i+1} \beta y y^\circ  = 1 = x^2 - \pi^{2i+1} \beta y^2 \quad \text{and} \quad \pi^{i} (x y^\circ  + {x}^\circ y) = 0.
$$
Suppose there exists
$
g = \smat{a}{b}{c}{d} \in \GU_2(\cO_\ell)$ such that $g^2 = \I_2  \text{ and }  gAg^{-1} =A^{-1}.$     Then we have
$$
    g^2 = \I_2 \iff a^2 + bc = 1, \quad a^2 = d^2, \quad \text{and} \quad b(a+d) = 0 = c(a+d)
$$
and
$$
    gA = A^{-1}g \iff \pi^{i} y (a+d) = 0 \quad \text{and} \quad \pi^{i} y (b + \pi \beta c) = 0.
$$

Let $a + d = \pi^{k} u$ and $b + \pi \beta c = \pi^{m} v$, where $0 \leq k, m \leq \ell$ and $u, v \in \Lri_\ell^{\times}$. Depending on the values of $k$ and $m$, we consider the following cases:

\textbf{Case 1:} $k = 0$ or $m = 0$. From $gA = A^{-1}g$, we obtain $\pi^{i} y = 0$. This implies that $A$ is a scalar matrix, a contradiction.

\textbf{Case 2:} $1 \leq k \leq \ell - 1$. From $b(a+d) = 0 = c(a+d)$, we have $b = \pi^{\ell - k} b'$ and $c = \pi^{\ell - k} c'$, where $b', c' \in \Lri_\ell^{\times}$. Since $1 \leq k$ and $1 \leq \ell - k$, the equations $a^2 + bc = 1$ and $a {d}^\circ + c {b}^\circ = 1$ yield
$$
    a^2 = 1 \pmod{\pi} \quad \text{and} \quad -a {a}^\circ = 1 \pmod{\pi}.
$$
This implies $a = -{a}^\circ \pmod{\pi}$, so $a = {\epsilon} a'$ for some $a' \in \lri_\ell^{\times}$. Since $a^2 = 1 \pmod{\pi}$, we deduce that $\epsilon^2 \mod (\pi)$ has a square root. Recall that the residue field is of odd characteristic. Therefore    $\nonsq^2 \in (\cO_\ell^\times)^2$ if and only if $\rho_{\ell,1}({{\nonsq}^2}) \in (\cO_1^\times)^2. $ Hence we get $\nonsq^2 \in (\cO_\ell^\times)^2 $, a contradiction.

\textbf{Case 3:} $k = \ell$ and $1 \leq m \leq \ell$. Then $a = -d$ and $b = \pi^{m} v - \pi \beta c$, where $v \in \Lri_\ell^{\times}$. The equations $a^2 + bc = 1$ and $a{d}^\circ + c{b}^\circ = 1$ lead to a contradiction as in Case 2.

Therefore, $A$ is not strongly real, and the proof is complete.
\end{proof}


  The following result provides a count of real and strongly real classes in $\mathrm{GU}_{2}(\mathfrak{o}_{\ell})$.
\begin{proposition}\label{cor-real-GU-count}
The following statements hold:
\begin{enumerate}
\item
The numbers of real regular and real non-regular classes in $\GU_{2}(\mathfrak{o}_{\ell})$
coincide with those in $\GL_{2}(\mathfrak{o}_{\ell})$; hence, the two groups have the
same total number of real classes which is $ 1 + q^{\ell} + 2 \sum_{0}^{\ell-1} q^{i}.$

\item
In $\GU_{2}(\mathfrak{o}_{\ell})$, the number of strongly real classes equals the number
of real regular classes.

\item
The numbers of real regular and real non-regular classes in $\G(\co_{\ell})$ are
given by
\[
q^{\ell}+1
\quad \text{and} \quad
2\sum_{i=0}^{\ell-1} q^{i},
\]
respectively.
\end{enumerate}
\end{proposition}

\begin{proof} We consider the following cases to count the total numbers of real regular and real non-regular classes in $\mathrm{GU}_{2}(\mathfrak{o}_{\ell})$, using the characterization of real classes obtained in \autoref{thm-real-GU-main}.

\textbf{Case (i):} Let $A = \smat{x}{0}{0}{x}$, where $x = \pm 1$. Both of these real classes are non-regular.

\textbf{Case (ii):} Let $A = \smat{x}{0}{0}{x^{-1}}$, where $x \in \cO_\ell^{\times}$ such that $x^2 \neq 1$. The total number of distinct real  classes in this case is
$$
\frac{|\cO_\ell^{\times}| -2}{2}  = \frac{q^\ell - q^{\ell-1}}{2} - 1.
$$
Moreover, $A$ is regular if and only if $x^2 - 1$ is a unit. Therefore, the total number of distinct real regular and real non-regular classes is given by
$$
\frac{(q-3)q^{\ell-1}}{2} \quad \text{and} \quad q^{\ell-1} - 1, \quad \text{respectively.}
$$

\textbf{Case (iii):} Let $A = \smat{0}{1}{1}{0}$ or $A = \smat{x}{y}{y}{x}$, where $y \neq 0$ satisfies $xx^\circ + yy^\circ = 1$, $xy^\circ + x^\circ y = 0$, and $x^2 - y^2 = 1$. For $A = \smat{x}{y}{y}{x}$, by using $x^2 - y^2 = 1$, $xx^\circ + yy^\circ = 1$, and $xy^\circ + x^\circ y = 0$, we get that $x + y$ and $x - y$ are eigenvalues of $A$ with norm $1$. This implies
$$
x = a \text{ and } y = b\epsilon \text{ such that } a^2 - b^2 \epsilon^2 = 1,
$$
where $a, b \in \co_\ell$ and $b \neq 0$. Let the norm map $\mathcal{N}: \Lri_\ell^{\times} \rightarrow \co_{\ell}^{\times}$ be defined by $\mathcal{N}(a + b\epsilon) = a^2 - b^2 \epsilon^2$ for all $a, b \in \cO_\ell$. Then $\mathcal{N}$ is surjective and its kernel has order $q^{\ell} + q^{\ell-1}$. Therefore, the total number of matrices $\smat{a}{b\epsilon}{b\epsilon}{a}$ with $a, b \in \co_\ell$ and $b \neq 0$ such that $a^2 - b^2 \epsilon^2 = 1$ is $q^\ell + q^{\ell-1} - 2$.

Note that $\smat{0}{1}{1}{0}$ is a regular element with determinant $-1$, and $\smat{a}{b\epsilon}{b\epsilon}{a}$ is regular if and only if $b$ is a unit. Furthermore, there are $q^{\ell-1}-1$ choices of nonzero $b \in \pi\mathfrak{o}_\ell$, and for each such $b$, there are exactly two solutions $a\in \cO_\ell$ of $a^2 - b^2 \epsilon^2 = 1$ obtained by lifting from $a \equiv \pm 1 \pmod{\pi}$. 
Therefore, among the matrices $\smat{a}{b\epsilon}{b\epsilon}{a}$ satisfying $a^2-b^2\epsilon^2=1$ with $a\in\cO_\ell$, the number with $b\in \pi\mathfrak{o}_\ell\setminus\{0\}$ is $2\bigl(q^{\ell-1}-1\bigr)$, while the number with $b\in\mathfrak{o}_\ell^\times$ is
$$
(q^\ell+q^{\ell-1}-2)-2\bigl(q^{\ell-1}-1\bigr)
= q^\ell-q^{\ell-1}
= (q-1)q^{\ell-1}.
$$

Since the norm map $\mathcal{N}$ is surjective, we can choose $\delta \in \Lri_\ell^{\times}$ such that $\delta\delta^\circ = -1$. Using $\smat{\delta}{0}{0}{-\delta}$, we see that $\smat{0}{1}{1}{0}$ and $\smat{x}{y}{y}{x}$ are conjugate in $\mathrm{GU}_2(\mathfrak{o}_\ell)$ to $\smat{0}{-1}{-1}{0}$ and $\smat{x}{-y}{-y}{x}$, respectively. Hence, the total number of distinct real regular, real non-regular, and real classes in this case is given by
$$
\frac{(q-1) q^{\ell-1}}{2} + 1, \quad q^{\ell-1} - 1, \quad \text{and} \quad \frac{q^\ell + q^{\ell-1} - 2}{2} + 1,
$$
respectively.

\textbf{Case (iv):} Let $A = \smat{x}{\pi^{i+1} \beta y}{ \pi^{i} y }{x}$, where $0\leq i \leq \ell-1$, $\beta \in \cO_{\ell-i-1}$, $x \in \cO_\ell^{\times}$, and $y \in \Lri_\ell^{\times}$ such that $x^2 + \pi^{2i+1} \beta y {y}^\circ = 1 = x^2 - \pi^{2i+1} \beta y^2$ and $\pi^i(x{y}^\circ + {x}^\circ y) = 0$.

Using the equation $\pi^{i} (x{y}^\circ +{x}^\circ y) = 0$, we obtain $\pi^{i}y = \pi^{i} r\epsilon x$, where $r \in \cO_{\ell-i}^{\times}$. Therefore, $A$ has the following form:
\begin{equation}\label{eq-real-class-unitary-type-4}
A = \smat{x}{\pi^{i+1}r\epsilon \beta x}{ \pi^{i}r\epsilon x}{x},
\end{equation}
where $x \in \cO_\ell^{\times}$, $r \in \cO_{\ell-i}^{\times}$, and $\beta \in \cO_{\ell-i-1}$ such that $x^2(1 - \pi^{2i+1} r^2 \epsilon^2 \beta) = 1$.

To count the distinct conjugacy classes of matrices given in \eqref{eq-real-class-unitary-type-4}, we claim that the conjugacy classes depend only on $i$, $x$, and $\beta$. To see this, note that for fixed $i$, for each choice of $r^2\beta$, there are exactly two values of $x$ satisfying $x^2(1 - \pi^{2i+1} r^2 \epsilon^2 \beta) = 1$.
Furthermore, since $r^2 \in \cO_{\ell-i}^{\times}$ and $\ell-i-1 < \ell-i$, the map $\beta \mapsto r^2\beta$ is a bijection on $\cO_{\ell-i-1}^{\times}$. Therefore, the total number of choices for $r^2\beta$ equals that of $\beta$, which is $|\cO_{\ell-i-1}|$. 
Hence, the total number of distinct  real classes is
$$
2\left( \sum_{i=0}^{\ell-1} |\cO_{\ell-i-1}| \right) = 2\left( \sum_{j=0}^{\ell-1} q^j \right).
$$

Note that $A$ is regular if and only if $i = 0$. Therefore, the total number of distinct real regular and real non-regular  classes of this form is
$$
2q^{\ell-1} \quad \text{and} \quad 2\left( \sum_{i=0}^{\ell-2} q^i \right), \quad \text{respectively}.
$$

From the above four cases, we conclude that the total number of real regular, real non-regular, and real classes is
$$
q^\ell + 1, \quad 2\sum_{i=0}^{\ell-1} q^i, \quad \text{and} \quad 1 + q^\ell + 2 \sum_{i=0}^{\ell-1} q^{i}, \quad \text{respectively}.
$$
Furthermore, \autoref{cor-str-real-GU} implies that the total number of strongly real classes is
$
q^\ell + 1.
$
The proof now follows from the above observations, \autoref{thm-real-GL-main}, and \autoref{cor-real-GL-count}. \qedhere

\end{proof}

We can summarize the main observations of the above proof in the following table: 
  \begin{tiny}
\begin{table}[h]
    \centering
    \caption{Number of real regular, real non-regular, and strongly real classes in $\mathrm{GU}_{2}(\mathfrak{o}_{\ell})$}
    
  \label{tab:real_classes}
    \renewcommand{\arraystretch}{1.4}
    \begin{tabular}{|c|m{7cm}|c|c|c|}
        \hline
        \textbf{Type} & \textbf{Real classes representative} & \textbf{Real regular} & \textbf{Real non-regular} & \textbf{Strongly real} \\
        \hline
        1 & $\begin{pmatrix} x & 0 \\ 0 & x \end{pmatrix},\ x = \pm 1$ & $0$ & $2$ & $2$ \\
        \hline
        2 & $\begin{pmatrix} x & 0 \\ 0 & x^{-1} \end{pmatrix},\ x \in \cO_\ell^{\times},\ x^2 \neq 1$ 
          & $\dfrac{(q-3)\,q^{\ell-1}}{2}$ & $q^{\ell-1} - 1$ & $\dfrac{(q-1)\,q^{\ell-1}}{2} - 1$ \\
        \hline
        3 & $\begin{pmatrix} 0 & 1 \\ 1 & 0 \end{pmatrix}$ 
          & $1$ & $0$ & $1$ \\
        \hline
        4 & 
        $\begin{aligned}
            &\begin{pmatrix}a & b \epsilon \\ b\epsilon & a\end{pmatrix},\ a \in \cO_\ell,\ b \in \cO_\ell^{\times}, \\
            &\quad\mathcal{N}(a + b\epsilon) = 1
        \end{aligned}$
          & $\dfrac{(q-1)\,q^{\ell-1}}{2}$ & $q^{\ell-1} - 1$ & $\dfrac{(q+1)\,q^{\ell-1}}{2} - 1$ \\
        \hline
        5 & 
        $\begin{aligned}
            &\begin{pmatrix} x & \pi^{i+1} \beta y \\ \pi^{i} y & x \end{pmatrix},\ 0\leq i \leq \ell-1,\, \beta \in \cO_{\ell-i-1}, \\
            &x \in \cO_{\ell}^{\times},\ y \in \Lri_\ell^{\times},\ x^2 - \pi^{2i+1} \beta y^2 = 1
        \end{aligned}$
          & $2 q^{\ell-1}$ & $2\sum\limits_{i=0}^{\ell-2} q^i$ & $0$ \\
        \hline
        \hline
        \textbf{Total} & $1 + q^\ell + 2\displaystyle\sum_{i=0}^{\ell-1} q^{i}$ & $q^\ell + 1$ & $2\displaystyle\sum_{i=0}^{\ell-1} q^i$ & $q^\ell + 1$ \\
        \hline
    \end{tabular}
\end{table}
\end{tiny}


\section{Real and orthogonal characters of $\G(\cO_\ell)$ for $\ell \geq 2$}

Recall that an irreducible representation $\sigma$ of a finite group $G$ has a real character if and only if $\sigma$ is self-dual. It is well known that a representation $\sigma$ of $G$ is self-dual if and only if $\langle \sigma \otimes \sigma, \mathbf{1}_G \rangle \neq 0$. If $\sigma$ is irreducible, then $\langle \sigma \otimes \sigma, \mathbf{1}_G \rangle \leq 1$. Consequently, an irreducible representation $\sigma$ has a real character if and only if $\langle \sigma \otimes \sigma, \mathbf{1}_G \rangle = 1$. We shall use this criterion to determine the real characters of $\G(\cO_\ell)$.

Our objective is to characterize all regular and non-regular irreducible representations $\sigma$ of $\G(\cO_\ell)$ satisfying $\langle \sigma \otimes \sigma, \mathbf{1}_{\G(\cO_\ell)} \rangle = 1$. The next lemma describes the real non-regular characters of $\G(\cO_\ell)$.
  
\begin{lemma}
\label{lem:non-regular real characters}
    A non-regular representation $\sigma \in \irr(\G(\co_\ell))$ has real character if and only if there exists $\sigma' \in \Irr(\G(\co_{\ell-1}))$ with real character such that $\sigma=\sigma'  \circ \rho_{\ell,\ell-1} $. 
\end{lemma}

\begin{proof}
It is clear that if there exists $\sigma' \in \Irr(\G(\co_{\ell-1}))$ with real character such that $\sigma=\sigma'  \circ \rho_{\ell,\ell-1} $ then $\sigma$ has real character. We prove the converse. Assume that $\sigma \in \Irr(\G(\co_\ell))$ is a non-regular representation with real character. Then $\sigma _{|_{\K^{\ell-1}}}$ must be a real character of $\K^{\ell-1}$. By definition of non-regular representations we have $\sigma _{|_{\K^{\ell-1}}}=\dim{(\sigma)}\psi_A$ for a scalar matrix $A\in \g(\co_1)$. Hence $\psi_A$ has real character. Since $\K^{\ell-1}$ is an odd order group, only real character of $\K^{\ell-1}$ is trivial. This implies $\psi_A=\mathbf{1}_{\K^{\ell-1}}$. Hence $\sigma=\sigma'  \circ \rho_{\ell,\ell-1} $ for some $\sigma' \in \Irr(\G(\co_{\ell-1}))$.  
\end{proof}
Hence, it boils down to investigate real regular characters of $\G(\co_\ell)$. For any regular $A \in \g(\cO_m),$ consider the determinant map $\det: \C_{\G(\cO_m)}(A) \rightarrow R_m^\times$. For $A \in \g(\cO_{\ldown})$, define 
\[\mathrm{Z}_A = \{x\I \in \G(\cO_\ell) \mid  x \in \rl^\times, \, \rho_{\ell,\ldown}(x)\in \det(\C_{\G(\co_{\ldown})}(A))\}. \]
Then $\mathrm{Z}_A$ is a central subgroup of $\G(\cO_\ell)$ determined by $A$. The following is true regarding the group $\mathrm{Z}_A.$

\begin{lemma}\label{lem:preimage-det-map} The following hold for regular $A \in \g(\cO_{\ldown}):$
\begin{enumerate}
    \item For $A$ such that $\tt(A) \in \{\ss, \cus\},$ we have $\mathrm{Z}_A= \Z_\ell .$ 
    \item For $A$ such that $\tt(A) \in \{\sns\},$ the group $\mathrm{Z}_A$ is an index two subgroup of $\Z_\ell .$ 
\end{enumerate}
\end{lemma}   
\begin{proof}
Define a set $S_\ell\subseteq R_{\ell}^\times$ as follows
\[
S_\ell=\begin{cases}
    \co_{\ell}^\times & \text{for } \G=\GL_2\\
    \{z \in \Lri_\ell^\times \mid zz^\circ=1\} & \text{for } \G=\GU_2.
\end{cases}
\]

Define $z_A=\{x\in S_\ell \mid \rho_{\ell,\ldown}(x)\in \det(\C_{\G(\co_{\ldown})}(A))\}.$ Since $\det(g)\in S_\ell$ for any $g\in \G(\co_\ell)$ and the projection map from $S_\ell$ to $S_{\ldown}$ is surjective, the set $z_A$ is well defined. Then it is easy to observe that  $[\Z_\ell :Z_A]=[S_\ell : z_A].$ Since $\ker(\rho_{\ell,\ldown}\colon S_\ell \rightarrow S_\ldown)\subseteq z_A,$ we have $=[S_\ell:z_A]=[S_{\ldown} : \det(\C_{\G(\co_{\ldown})}(A))].$ Now we claim that $[S_{\ldown}: \det(\C_{\G(\co_{\ldown})}(A))]=[S_1: \rho_{\ell,1}(\det(\C_{\G(\co_{\ldown})}(A)))]$. For proving this it is enough to prove that $\ker(\rho_{\ldown,1}\colon S_{\ldown}\rightarrow S_1)\subseteq \det(\C_{\G(\co_{\ldown})}(A)).$ Consider the element $1+\pi w\in S_\ldown.$ Then by Hensel's lemma, the equation $1+\pi w = x^2 $ has a solution for some $x\in R_\ldown^\times$ where $x^2\in S_\ldown $. For $\G=\GL_2$ we have $x^2=\det(x\I)\in \det(\C_{\G(\co_{\ldown})}(A))$ and hence $\ker(\rho_{\ldown,1}\colon S_{\ldown}\rightarrow S_1)\subseteq \det(\C_{\G(\co_{\ldown})}(A)).$ For $\G=\GU_2,$ $1+\pi w=x^2$ gives $x=\pm 1 \mod(\pi)$ and since $1+\pi w\in S_{\ldown}, $ we also have $\mathcal{N}(x)=\pm 1.$ These two observations imply that $\mathcal{N}(x)=1$ which means $x\in S_\ldown.$ Hence $x^2=\det(x\I)\in \det(\C_{\GU_2(\co_{\ldown})}(A))$ which proves that $\ker(\rho_{\ldown,1}\colon S_{\ldown}\rightarrow S_1)\subseteq \det(\C_{\G(\co_{\ldown})}(A)).$ 
This proves the claim. By \cite[Proposition~4.5]{MR3737836}, we have $[S_1 : \rho_{\ell,1}(\det(\C_{\G(\co_{\ldown})}(A)))]=1$ when $\tt(A)=\ss, \cus$ and $[S_1 : \rho_{\ell,1}(\det(\C_{\G(\co_{\ldown})}(A)))]=2$ if $\tt(A)=\sns.$ This proves the lemma.

\end{proof}

\begin{definition}(Tangible representation) An irreducible regular representation $\sigma$ of $\G(\co_\ell)$ for $\ell \ge 2$ is called a ${\it \tan}$ representation if the following hold for any $A \in \g(\cO_\ldown)$ satisfying $\langle \sigma, \psi_A \rangle_{\K^\lup} \neq 0:$
\begin{enumerate}
\item[(T1)] $A + g A g^{-1} = 0 $ for some $g \in \G(\cO_\ell).$  
\item[(T2)] $\langle \sigma, {\bf 1} \rangle_{\mathrm{Z}_A} \neq 0. $
\end{enumerate}   
\end{definition}

\begin{lemma}
\label{lem:real-char-numbers-total}

Let $\ell \geq 2$ and $\mathbf{m}_\star $ denote the total number of $\tan$ representations of $\G(\co_\ell)$ with type $\star \in \{\ss, \sns\ \cus \}$. Then 
\[
\mathbf{m}_\star=\begin{cases}
     \frac{1}{2}q^{\ell-2}(q-1)^2; & \star=\ss, \\
    \frac{1}{2}q^{\ell-2}(q^2-1); & \star=\cus,\\
     2q^{\ell-1}; & \star=\sns.
 \end{cases}
 \]
\end{lemma}
\begin{proof}
Let $\sigma$ be a tangible representation and $A \in \g(\cO_\ldown)$ such that $\langle \sigma, \psi_A \rangle \neq 0$. By the definition of tangible representation, we obtain $\tr(A)=0\mod (\pi^{\ldown})$. Now onward, we assume that $A\in \g(\co_{\ldown})$ is of the form mentioned in \autoref{lem:orbit-representatives-gol} and satisfies $\tr(A)=0 \mod (\pi^{\ldown}).$ We consider $\ell$ even and $\ell $ odd cases separately.

{\bf $\ell$ even: } Let $\ell=2m$. By construction of regular representations for $\ell=2m$, we have $\sigma=\ind_{S_A}^{\G(\co_\ell)}\widetilde{\psi_A}$ for some one dimensional representation $\widetilde{\psi_A}\in \Irr(S_A \mid \psi_A)$. For tangible $\sigma$, the condition $\langle \sigma, {\bf 1} \rangle_{\mathrm{Z}_A} \neq 0$ is equivalent to $\widetilde{\psi_A}=\bf{1}$ on $\mathrm{Z}_A.$ We have the following:
\begin{enumerate} 
\item The number of characters lying above $\psi_A$ and satisfying $\widetilde{\psi_A}=\bf{1}$ on $\mathrm{Z}_A$ is given by $\frac{|S_A|}{|\mathrm{Z}_A\K^{m}|}$. 
\item The number of matrices $A\in\g(\co_m)$  which give rise to non-conjugate characters $\psi_A$ of $\K^m$ is $\frac{|\co_m^\times|}{2}$ for $\tt(A)=\ss, \cus$ and $|\co_{m-1}|$ for $\tt(A)=\sns.$
\end{enumerate} 
We note that $\mathbf{m}_{\star}=\text{number of non-conjugate $\psi_A$ of type $\star$ }\times\frac{|S_A|}{|\mathrm{Z}_A\K^{m}|}$ for even $\ell$. The result now follows by direct computation using \autoref{lem:cardinalities of centralizers} and \autoref{lem:preimage-det-map}.

{\bf $\ell$ odd: } We first consider $\mathbf{m}_\star$ such that $\star \in \{\sns, \ss \}.$ In this case, by \autoref{eqn:SA-representation-sns-ss}, any tangible character is of the form $\ind_{\N C_{\tilde{A}}}^{\G(\cO_\ell)}(\psi'_{\tilde{A}}).$

The number of distinct characters of $\N$ of the form $\psi_{\ti{A}}$ (as defined in \autoref{eqn: definition of psi A on N}) are $\frac{1}{2}(q^{\lup}-q^{\lup-1})$ and $q^{\lup-1}$ for $\tt(A) = \ss$ and $\tt(A) = \sns,$  respectively. Further, for a given $\psi_{\tilde{A}},$ the number of distinct $\psi'_{\tilde{A}}$ such that the condition $\langle \sigma, {\bf 1} \rangle_{\mathrm{Z}_A} \neq 0 $ is satisfied  is $\frac{|\N\cta|}{|\N\mathrm{Z}_A|}$. Since $\N \subseteq \K^{\ldown},$ we have $\frac{|\N\cta|}{|\N\mathrm{Z}_A|}=\frac{|\cta|}{| \mathrm{Z}_A |} \mod (\pi^{\ldown})$. Using this along with \autoref{lem:cardinalities of centralizers} and \autoref{lem:preimage-det-map}, we get that  $\mathbf{m}_{\sns}=2q^{\ell-1}$ and $\mathbf{m}_{\ss}=\frac{1}{2}q^{\ell-2}(q-1)^2.$
  
We now consider the case of $\star = \cus.$ The number of non-conjugate characters $\psi_A$ of $\K^{\lup}$ that contribute to tangible characters is $\frac{1}{2}(q^{\ldown}-q^{\ldown-1}).$ The number of tangible representations $\sigma$ lying above the character $\psi_A$ of $\K^{\lup}$ satisfying $\langle \sigma, {\bf 1} \rangle_{\mathrm{Z}_A} \neq 0$ is  $\frac{|S_A|}{|\Z_\ell  D^{\ldown}(\ti{A})|}\times \frac{|\Z_\ell  D^{\lup}(\ti{A})|}{|\Z_\ell  \K^{\lup}|}.$ Thus, 
\[
\mathbf{m}_{\cus} = \frac{|S_A|}{|\Z_\ell  D^{\ldown}(\ti{A})|}\times \frac{|\Z_\ell  D^{\lup}(\ti{A})|}{|\Z_\ell  \K^{\lup}|} \times \frac{1}{2}(q^{\ldown}-q^{\ldown-1}) =   \frac{1}{2}q^{\ell-2}(q^2-1). 
\]
\end{proof}
We are now in a position to characterize the real regular characters of $\G(\cO_\ell).$
\begin{theorem}
\label{thm:regular-self-dual}
For $\ell \geq 2,$ any regular irreducible representation of $\G(\cO_\ell)$  is self-dual if and only if it is tangible.

\end{theorem}

\begin{proof}
Let $\sigma$ be a regular self-dual representation of $\G(\cO_\ell)$. By the construction of the regular representations of $\G(\cO_\ell)$, there exists a character $\psi_A \in \K^\lup$ for regular $A \in \g(\cO_\ldown)$, and an irreducible representation $\phi\in \Irr(S_A \mid \psi_A)$ such that $\sigma \cong \ind_{S_A}^{\G(\cO_\ell)}(\phi).$  By Mackey's restriction theorem and Frobenius reciprocity we have,
    \begin{equation}
    \label{eq:Mackey}
\langle\sigma\otimes \sigma,\mathbf{1}\rangle=\langle \ind_{S_A}^{\G(\co_\ell)}\phi \otimes \ind_{S_A}^{\G(\co_\ell)}\phi , \mathbf{1}\rangle _{\G(\co_\ell)}=\underset{h\in S_A \backslash \G(\co_\ell) / S_A}{\oplus}{\langle \phi\otimes \phi^h, \mathbf{1}\rangle}_{S_A \cap S_A^h}
   \end{equation}
Since $\sigma$ is irreducible, we have $\langle\sigma\otimes \sigma,\mathbf{1}\rangle \leq 1$. Hence, $\sigma$ has real character if and only if there exists a unique  $g\in S_A \backslash \G(\co_\ell) / S_A$ such that 
\begin{eqnarray}
\label{eq:phi-condition-real}
{\langle \phi\otimes \phi^g, \mathbf{1}\rangle}_{S_A \cap S_A^g} & =  & 1.
\end{eqnarray} 
Since $\K^{\lup} \subseteq S_A \cap S_A^{g}$, we observe ${\langle \phi\otimes \phi^g, \mathbf{1}\rangle}_{S_A \cap S_A^g} = 1$ implies ${\langle \phi\otimes \phi^{g}, \mathbf{1}} \rangle_{\K^{\lup}} \neq 0 $. By the construction of regular representations of $\G(\co_\ell)$, we have ${\phi}_{\mid _{\K^{\lup}}}=q^{\lup-\ldown}\psi_A$. Therefore 
    \[
    {\langle \phi\otimes \phi^{g}, \mathbf{1}} \rangle_{\K^{\lup}} \neq 0 \,\, \mathrm{if \,\, and \,\,  only \,\,  if} \,\,  {\langle \psi_{A+gAg^{-1}},\mathbf{1}\rangle}_{\K^{\lup}} \neq 0.
    \]
 By definition of $\psi_A,$ we obtain $A+gAg^{-1}=0 \mod (\pi^{\ldown}).$ This proves (T1). 
 
 To prove (T2), we choose the Ad orbit representatives $A\in \g(\co_{\ldown})$ as given in \autoref{lem:orbit-representatives-gol}.
 Without loss of generality, we can fix $g= \smat{z}{0}{0}{-z} \in \G(\cO_\ell)$ for $A$ such that $\tt(A) \in \{ \sns, \cus\}$ and fix $g$ to be $\smat{0}{1}{1}{0} \in \G(\cO_\ell)$ for any $A$ with $\tt(A) = \ss.$ Now onward, we will assume that $g$ is of this form.  The following hold for any Serre lift $\tilde{A}$ of $A$ and above $g$:

\begin{enumerate}
    \item[(a)]  $\ti{A}+g\ti{A}g^{-1}=0 $.
    \item[(b)] $\mathrm{tr}(\ti{A})=0.$
    \item[(c)] $S_A \cap S_A^g=S_A$.
\end{enumerate}
We now consider the case of $\ell$ even and odd separately. 

\vspace{.2cm} 
\noindent {\bf 
$\ell$ even:} For this case, there exists one-dimensional representation $\widetilde{\psi_A}$ of $S_A$ such that $\sigma \cong \ind_{S_A}^{\G(\cO_\ell)}(\widetilde{\psi_A}).$
Therefore, by  Mackey's restriction formula, $\sigma$ is self-dual if and only if $\widetilde{\psi_A}\otimes \widetilde{\psi_A}^g|_{S_A \cap S_A^g} =\mathbf{1}_{S_A \cap S_A^g} $ for $g $ chosen as above. By (a)-(c),  we must have
  \begin{equation}
  \label{eqn:even case second eqn}
      \widetilde{\psi_A}\otimes \widetilde{\psi_A}^g =\mathbf{1}  \text{ on } S_A. 
  \end{equation}
Since $A+gAg^{-1}=0,$ we have  $\widetilde{\psi_A}\otimes \widetilde{\psi_A}^g = {\psi_A}\otimes {\psi_A}^g =\psi_{A+gAg^{-1}}=\mathbf{1}$ as characters of $\K^\lup$. Since $\ell$ is even, we have $S_A=\C_{\G(\co_\ell)}(\ti{A})\K^\lup$ for any lift $\ti{A}$ of $A$.  Hence \autoref{eqn:even case second eqn} is equivalent to
      \[
    \widetilde{\psi_A}\otimes \widetilde{\psi_A}^g ={\mathbf{1}} \text{ on } \C_{\G(\co_\ell)}(\ti{A}).
     \]
The above equation gives \[\widetilde{\psi_A}(x\I+y\ti{A})\widetilde{\psi_A}^g(x\I+y\ti{A})=1 \text{ for all } x\I+y\ti{A}\in \C_{\G(\co_\ell)}(\ti{A}) .\]
We assume that $\ti{A}$ is a Serre lift of $A$. By (a), we get 
\[
\widetilde{\psi_A}(x\I+y\ti{A})\widetilde{\psi_A}(x\I-y\ti{A})=1 \text{ for all } x\I+y\ti{A}\in \C_{\G(\co_\ell)}(\ti{A}) .\]
Therefore, we obtain  $\widetilde{\psi_A}(x^2\I-y^2 \ti{A}^2)=1$ for all $x\I+y\ti{A}\in \C_{\G(\co_\ell)}(\ti{A}).$ Now (2) follows by observing that  $\mathrm{Z}_A = \{ x^2 \I - y^2 \ti{A}^2 \mid x\I+y\ti{A}\in \C_{\G(\co_\ell)}(\ti{A})\}$ for $\tr(\ti{A}) = 0$.

To prove the converse, let $A\in \g(\cO_\ldown)$ be regular such that $\langle \sigma, \psi_A \rangle \neq 0 .$ Let  $\tilde{A}$ be a lift of $A$. By using $ \widetilde{\psi_A}(x^2\I-y^2 \ti{A}^2)=1$ for all $x,y$ such that $x\mathrm{I}+y\tilde{A} \in {\G(\cO_\ell)}$ and $A+gAg^{-1}=0 \mod(\pi^{\ldown})$ for some $g\in \G(\co_\ell)$ gives $ \widetilde{\psi_A}\otimes \widetilde{\psi_A}^g=\mathbf{1} $ on $S_A$.  Then we obtain that $\sigma$ is self-dual by \autoref{eq:Mackey}.
This proves \autoref{thm:regular-self-dual} for even $\ell.$ 

{\bf $\ell$ odd:}  

We first prove our result for the $\sns$ and $\ss$ case by using the results of \autoref{subsec:alternate const for ss,sns}. From the results of that section, we have that any split semisimple or split non-semisimple representation $\phi$ of $S_A$ satisfies $\phi \cong \ind_{\N\cta}^{S_A}\psi_{\ti{A}}'$  for Serre lift $\ti{A}$ of $A$. Therefore, using \autoref{eq:phi-condition-real}, we must have 
\begin{equation}
\label{eqn:sns eq2}
    \langle \ind_{\N\cta}^{S_A}{\psi_{\ti{A}}'} \otimes (\ind_{\N\cta}^{S_A}{\psi_{\ti{A}}'})^g,\mathbf{1}\rangle_{S_A} =1,
\end{equation}
  for $g$ as mentioned above. This is equivalent to
  \[
  \underset{h\in \N\cta \backslash \G(\co_\ell)/ (\N\cta)^g}{\oplus}\langle {\psi_{\ti{A}}'} \otimes ({\psi_{\ti{A}}'})^{hg},\mathbf{1}\rangle=1 \text{ on } {\N\cta\cap (\N\cta)^{gh}}.
  \]
  Since ${\psi_{\ti{A}}'}$ is one dimensional, the above equation is equivalent to
  \begin{equation}
  \label{eqn:common eqn odd}
      \underset{h\in \N\cta \backslash \G(\co_\ell)/ (\N\cta)^g}{\oplus} {\psi_{\ti{A}}'} \otimes ({\psi_{\ti{A}}'})^{hg}=\mathbf{1} \text{ on } {\N\cta\cap (\N\cta)^{gh}}.
  \end{equation}
 Take $A= \smat{{\alpha}}{0}{0}{-{\alpha}} \in \g(\co_{\ldown})$ such that $\tt(A)=\ss$ with $g=\smat{0}{1}{1}{0}$ and $A=\smat{0}{\nonsq\pi {\beta}}{\nonsq}{0} \in \g(\co_{\ldown})$ such that $\tt(A) = \sns,$ with $g=\smat{z}{0}{0}{-z}\in \G(\co_\ell)$. Using \autoref{lem:info about ss and sns odd case new lemma}, \autoref{eqn:common eqn odd} is equivalent to 
\begin{equation}
\label{eqn:final eqn sns}
   {\psi_{\ti{A}}'}\otimes {\psi_{\ti{A}}'}^g=\mathbf{1}  \text{ on }
\cta.
\end{equation}
Using the definition of $\ti{A}$ and $\psi_{\ti{A}}',$ this is equivalent to $\psi_{\ti{A}}' = \bf{1}$ on $\mathrm{Z}_A.$ Since $\sigma$ lies above $\psi_{\ti{A}}'$ and $\mathrm{Z}_A$ is central subgroup, we obtain (T2) for this case. 
Thus every regular self-dual representation of type either $\ss$ or $\sns$ is a tangible representation. 

We next prove that every tangible representation of type either $\ss$ or $\sns$ is self-dual. Let $A\in \g(\cO_\ldown)$ be regular such that $\langle \sigma, \psi_A \rangle \neq 0 .$ By \autoref{lem:orbit-representatives-gol}, we can assume $A= \smat{{\alpha}}{0}{0}{-{\alpha}} \in \g(\co_{\ldown})$ for $\tt(A)=\ss$ and $A=\smat{0}{\nonsq\pi {\beta}}{\nonsq}{0} \in \g(\co_{\ldown})$ for $\tt(A) = \sns.$ Further, any representation $\phi$ of $S_A$ satisfies $\phi \cong \ind_{\N\cta}^{S_A}\psi_{\ti{A}}'$  for some lift $\ti{A}$ of $A$ mentioned in \autoref{subsec:alternate const for ss,sns}. Therefore, we must have $\psi_{\ti{A}}' = \bf{1}$ on $\mathrm{Z}_A.$ Hence $\psi'_{\tilde{A}}$ satisfies \autoref{eqn:final eqn sns} (and hence \autoref{eq:phi-condition-real}) for $g=\smat{0}{1}{1}{0}$ when $\tt(A)=\ss$ and $g=\smat{z}{0}{0}{-z}\in \G(\co_\ell)$ when $\tt(A)=\sns$. This implies that $\sigma$ is self-dual.

At last we consider the case of cuspidal regular representations. Let $\sigma$ be a tangible cuspidal representation. By \autoref{lem:orbit-representatives-gol}, there exists $A = \smat{0}{\nonsq \tilde{\alpha}}{\nonsq}{0}\in \g(\co_{\ldown}) $ such that $\tt(A) = \cus$ and $\langle \sigma, \psi_A \rangle_{\K^\lup} \neq 0.$ To prove that $\sigma \cong \ind_{S_A}^{\G(\cO_\ell)}(\phi)$ is self-dual, we prove that for $g=\smat{z}{0}{0}{-z}\in \G(\co_\ell),$ $\langle\phi\otimes\phi^g,\mathbf{1}\rangle=1$ on $S_A$. By (T2) and the fact that $\mathrm{Z_A}=\Z_\ell $ for $\tt(A)=\cus$ (\autoref{lem:preimage-det-map}) we get $\phi(z)=\I$ for all $z\in{\Z_\ell }$. 
We have 
\begin{equation}
      \label{eqn:cus phi tensor phi summation}     {\langle\phi\otimes\phi^g,\mathbf{1}\rangle}_{S_A}=\frac{1}{|S_A|} \sum_{h\in S_A}\chi_{\phi}(h)\chi_{\phi}(g^{-1}hg).
\end{equation}
By definition of $S_A$, $\Z_\ell  D^{\ldown}$ and $\Z_\ell  D^{\lup}$,    we have $S_A=\Z_\ell  D^{\lup}(\ti{A})\sqcup(\Z_\ell  D^{\ldown}(\ti{A})\setminus\Z_\ell  D^{\lup}(\ti{A}))\sqcup (S_A\setminus \Z_\ell  D^{\ldown}(\ti{A})).$ By \autoref{prop:character values}(2) we have 
\begin{equation}
    \label{eqn:chi=0 on SA minus ZDl1}
    \chi_\phi(h)=0 \quad \text{for all}\quad h\in \Z_\ell  D^{\ldown}(\ti{A})\setminus\Z_\ell  D^{\lup}(\ti{A}).
\end{equation} 
Now we claim that 

\begin{align}
\label{eq:dl2-condition}
\sum_{h\in \Z_\ell  D^{\lup}(\ti{A})}\chi_{\phi}(h)\chi_{\phi}(g^{-1}hg)=q^2|\Z_\ell  D^{\lup}(\ti{A})|.
\end{align} 

\noindent To prove this, consider $h=zab\in \Z_\ell  D^{\lup}(\ti{A})$ where $z\in \Z_\ell , a\in \K^1\cap S_A$ and $b\in \K^{\lup}$. Then, using $\phi(z)=\I$ we have 
\begin{gather*}
    \phi(h)=\phi(zab)=\phi(a)\phi(b).
\end{gather*}
By \autoref{prop:character values}(1), we have, $\phi(a)=\widetilde{\psi_A}(a)\I.$ Then by the above equation, $\phi(h)=\widetilde{\psi_A}(a)\phi(b).$ This implies $\chi_\phi(h)=\widetilde{\psi_A}(a)\chi_{\phi}(b).$ Again by \autoref{prop:character values}(1) we have $\chi_{\phi}(b)=q\widetilde{\psi_A}(b)=q\psi_A(b)$, since $b\in \K^{\lup}.$ This gives $\chi_\phi(h)=q\widetilde{\psi_A}(a)\psi_A(b).$ Using similar arguments we can easily deduce that $\chi_\phi(g^{-1}hg)=q\widetilde{\psi_A}(g^{-1}ag)\psi_A(g^{-1}bg).$ Let $a=\I+\pi\smat{x}{\nonsq y\ti{\alpha}}{\nonsq y}{x}\in \K^1\cap S_A$ and $b=\I+\pi^{\lup}B\in \K^{\lup}.$ By this we get
\begin{align*}
\chi_\phi(h)\chi_\phi(g^{-1}hg)=&q^2\widetilde{\psi_A}(a)\widetilde{\psi_A}(g^{-1}ag)\psi_A(b)\psi_A(g^{-1}bg)\\
    =& q^2 \widetilde{\psi_A}(ag^{-1}ag)\psi_A(bg^{-1}bg)\\
    =&q^2 \widetilde{\psi_A}(((1+\pi x)^2-\pi^2\nonsq^2 y^2 \ti{\alpha})\I)\psi(\pi^{\lup}\tr(\ti{A}+g^{-1}\ti{A}g)B)=q^2.
\end{align*}  
    Last equality follows because $\widetilde{\psi_A}(((1+\pi x)^2-\pi^2\nonsq^2 y^2 \ti{\alpha})\I)=1 $ (using $\phi(z)=\I$ for all $z\in \Z_\ell $) and $\ti{A}+g\ti{A}g^{-1}=0\mod(\pi^{\ldown})$, where $\ti{A}$ is the Serre lift of $A.$ Hence we get, $\sum_{h\in \Z_\ell  D^{\lup}(\ti{A})}\chi_{\phi}(h)\chi_{\phi}(g^{-1}hg)=q^2|\Z_\ell  D^{\lup}(\ti{A})|.$ Next we claim that 
    \begin{align}
        \label{eq-zdl1-condition}
    \sum_{h\in S_A\setminus \Z_\ell  D^{\ldown}(\ti{A})}\chi_{\phi}(h)\chi_{\phi}(g^{-1}hg)=|S_A\setminus \Z_\ell  D^{\ldown}(\ti{A})|.
    \end{align}
    To prove this, let $h=x\I+y\ti{A}\in S_A\setminus \Z_\ell  D^{\ldown}(\ti{A}).$ We have $(x\I+y\ti{A})(x\I-y\ti{A})=(x^2-y^2\ti{\alpha})\I.$ This gives $\phi(x\I+y\ti{A})(x\I-y\ti{A})=\phi((x^2-y^2\ti{\alpha})\I)=\I$. This implies $\phi(x\I-y\ti{A})=\phi(x\I+y\ti{A})^{-1}.$ Since $\chi_{\phi}(h)\chi_{\phi}(g^{-1}hg)=\chi_\phi(x\I+y\ti{A})\chi_{\phi}(x\I-y\ti{A}),$ we have $\chi_{\phi}(h)\chi_{\phi}(g^{-1}hg)=\chi_\phi(x\I+y\ti{A})\chi_{\phi}(x\I+y\ti{A})^{-1}.$ By \autoref{prop:character values}(3), $\chi_\phi(x\I+y\ti{A})^{-1}=\overline{\chi_\phi(x\I+y\ti{A})}=1$. This gives $\chi_{\phi}(h)\chi_{\phi}(g^{-1}hg)=\chi_\phi(x\I+y\ti{A})\overline{\chi_{\phi}(x\I+y\ti{A})}=1.$ This proves our claim. Further, $|\Z_\ell  D^{\ell_i}(\tilde{A})|=(q+\Delta) q^{4\ell -2\ell_i -2}$ for $i \in \{1,2\}$ are easy to prove for $\G = \GL_2$ and follow from  \cite[Section~4.H.2, Pages~53--54]{Campbell-thesis} for $\G = \GU_2$. 
    Using \autoref{eqn:chi=0 on SA minus ZDl1}, \autoref{eq:dl2-condition}, \autoref{eq-zdl1-condition} and $\frac{|\Z_\ell  D^{\ldown}(\ti{A})|}{|\Z_\ell  D^{\lup}(\ti{A})|}=q^2$
     in \autoref{eqn:cus phi tensor phi summation} gives $\langle \phi \otimes \phi^g, {\bf 1} \rangle_{S_A} = 1.$ Therefore,   $\sigma $ is self-dual. This proves that any tangible cuspidal representation of $\G(\cO_\ell)$ for $\ell$ odd is self-dual.

To prove that every self-dual cuspidal representation is tangible,  we show that the total number of self-dual cuspidal representations of $\G(\cO_\ell)$ is the same as the number of the tangible cuspidal representations (denoted as $\mathbf{m}_\cus$ in \autoref{lem:real-char-numbers-total}).   Denote by $\mathbf{n}_\star$, the number of self-dual representations of $\G(\co_\ell)$ of type $\star.$ We have already proved that every $\ss(\sns)$ representation is self-dual if and only if it is tangible. Therefore, by \autoref{lem:real-char-numbers-total}, we have  $\mathbf{n}_\ss=\frac{1}{2}(q-1)^2q^{\ell-2}$ and $\mathbf{n}_\sns=2q^{\ell-1}.$

   By \autoref{lem:non-regular real characters}, the number of non-regular self-dual representations (denote by $\mathbf{n}_{\nreg}$) of $\G(\co_{\ell})$ is equal to the number of self-dual representations of $\G(\co_{\ell-1})$. By \autoref{cor-real-GU-count},  $ \mathbf{n}_{\nreg}= 1 + q^{\ell-1} + 2\sum_{i=0}^{\ell-2} q^{i}.$ Since the number of self-dual representations of $\G(\cO_\ell)$ is same as the number of real classes of $\G(\cO_\ell),$ by \autoref{cor-real-GU-count}, we have  
   \[
   \mathbf{n}_{\ss}+\mathbf{n}_{\sns}+\mathbf{n}_{\cus}+\mathbf{n}_{\nreg}=1 + q^{\ell} + 2\sum_{i=0}^{\ell-1} q^{i}
   \]
 By substituting $\mathbf{n}_\ss, \mathbf{n}_\sns$ and $\mathbf{n}_\nreg,$ we get $\mathbf{n}_\cus=\frac{1}{2}(q^2-1)q^{\ell-1} = \mathbf{m}_\cus.$ Thus every self-dual cuspidal representation of $\G(\cO_\ell)$ is tangible. This completes the proof of the result.  
 \end{proof}

\section{Orthogonal characters} \label{sec-orth-char}

In this section, for $\ell \geq 2$, we prove that every real character of $\GL_2(\cO_\ell)$ is orthogonal, and that there exists a real character of $\GU_2(\cO_\ell)$ that is not orthogonal.


 For a representation $\phi$ of a group $G$, let $\upsilon(\phi)$ denotes its Frobenius-Schur indicator. Recall that an irreducible representation $\phi$ is self-dual and orthogonal if and only if $\upsilon(\phi)$ is $ \pm 1$ and $1$, respectively. We shall use the following well known result to determine if every real character is orthogonal:
\begin{equation}
\label{eqn:schur indicator formula}
    \sum_{\phi\in \Irr(G)} \upsilon(\phi)\dim(\phi)= |\{g \in G \mid g^2 = e\}|.
\end{equation}
Recall that an element $g\in \G(\co_\ell)$ is called an involution if $g^2=e$. We first determine the number of involutions in $\G(\cO_\ell).$

\begin{lemma}
\label{no. of involutions}
    The number of involutions in $\G(\co_\ell)$ is $(q - \partial_{\G})q^{2\ell-1}+2$.
\end{lemma}
\begin{proof}
The notion of involution is a conjugacy class invariant, hence to determine the involutions in $\G(\cO_\ell),$ we use the conjugacy class representatives of $\G(\cO_\ell)$ as given in \autoref{thm:conjugacy classes GL2} and  \autoref{lem-conjugacy-unitary}.

 By \autoref{thm:conjugacy classes GL2}, any element of $\GL_2(\cO_\ell)$ is conjugate to a matrix $M(d,i,\alpha,\beta),$ with certain conditions on $d,i, \alpha, \beta.$ The matrix $M(d,i,\alpha,\beta)$ is an involution if and only if the following two conditions are satisfied:
\begin{gather}
\label{eqn:GL2-involution-conditions}
   2 \pi^i d+\pi^{2i} \beta=0,\\
    d^2+\pi^{2i}\alpha=1.
\end{gather}
For $i=0,$ the above conditions give $d=\frac{-\beta}{2}$ and $\alpha=1-\frac{\beta^2}{4}.$ This gives a class representative $\smat{\frac{-\beta}{2}}{1-\frac{\beta^2}{4}}{1}{\frac{\beta}{2}}$ and it represents the conjugacy class $\smat{0}{1}{1}{0}$ for all $\beta.$ The total number of involutions of this class type are $\frac{|\GL_2(\co_\ell)|}{|\C_{\GL_2(\co_\ell)}(g)|}=\frac{q^{4\ell-3}(q-1)(q^2-1)}{(q-1)^2q^{2\ell-2}}=(q+1)q^{2\ell-1}$.
For $i\geq 1, $ it is easy to see from \autoref{eqn:GL2-involution-conditions} that $i=\ell.$ Using this in $d^2+\pi^{2i}\alpha=1$ gives $d=\pm 1.$ Hence we get two distinct conjugacy class representatives $\pm \smat{ 1}{0}{0}{ 1}$ and the total number of involutions of this class type is 2. From the above, we get $(q+1)q^{2\ell-1}+2$ number of involutions in $\GL_2(\co_\ell).$

Next we consider the conjugacy class representatives of the group $\GU_2(\co_\ell)$ as given in \autoref{lem-conjugacy-unitary}.

 \begin{enumerate}
    \item $g=\smat{x}{0}{0}{x}\in \GU_2(\co_\ell)$ : In this case $g^2=\smat{1}{0}{0}{1}$ implies $x=\pm 1$. Hence total number of involutions from these type of conjugacy classes is $2.$
    \item $g=\smat{x}{0}{0}{y}\in \GU_2(\co_\ell)$, $x\neq y$ : $g^2=\smat{1}{0}{0}{1}$ implies $x^2=y^2= 1$. The condition $x \neq y$ and  $x^2=y^2=1$ does not give any involution in this class type.
    \item $g=\smat{x}{ y}{y}{x}\in\GU_2(\co_\ell)$, $y\neq 0$: In this case, $g^2=\smat{1}{0}{0}{1}$ if and only if $xy=0$ and $x^2+y^2=1$. Since $y\neq 0$, using $xy=0$ we get either $y\in \Lri_\ell^\times$ and $x=0$ or $y\in \pi \Lri_\ell$ and $x\in \pi\Lri_\ell$. The case $y\in \pi \Lri_\ell$ and $x\in \pi\Lri_\ell$ is not possible because $\det(g)\neq 0 \mod (\pi)$. Hence, we must have $x =0$. Therefore $x^2+y^2=1$ gives $y^{2}=1$. The corresponding unique class representative is $g=\smat{0}{1}{1}{0}$. This class has total $\frac{|\GU_2(\co_\ell)|}{|\C_{\GU_2(\co_\ell)}(g)|}=\frac{q^{4\ell-3}(q-1)(q+1)^2}{(q+1)^2 q^{2\ell-2}}=(q-1)q^{2\ell-1}$ involutions.
    \item $g=\smat{x}{\pi^{i+1}\beta y}{\pi^i y}{x}\in \GU_2(\co_\ell), y \in \Lri_\ell^\times$ : In this case, $g^2=\smat{1}{0}{0}{1}$ gives $ \pi^{i+1}xy=0$. Since $y\in \Lri_\ell^\times$,  we get $x\in \pi \Lri_\ell$, which is not possible since $\det(g)\neq 0 \mod (\pi)$. Hence this class type does not give any involution.
    
\end{enumerate}

\noindent From the above discussion, it is clear that the total number of involutions in $\GU_2(\co_\ell)$ is $(q-1)q^{2\ell-1}+2$. This proves the Lemma.
\end{proof}

\begin{theorem}
\label{thm:orthogonal}
    Let $\ell \ge 1$. The following hold:
    \begin{enumerate} 
    \item Any irreducible complex representation of $\GL_2(\cO_\ell)$ with real-valued character is orthogonal.
    \item There exists a self-dual representation of $\GU_2(\cO_\ell)$ that is not orthogonal. For $\ell \geq 2$, there exists a regular symplectic          representaton. 
    \end{enumerate} 
\end{theorem}
\begin{proof}
To prove (1),  we show that the total sum of the dimensions of the self-dual representations  of $\GL_2(\cO_\ell)$ is the same as the number of involutions in $\GL_2(\cO_\ell).$ The result then follows from \autoref{eqn:schur indicator formula}.

 We proceed by induction. It is a well known result for $\ell = 1$, see \cite{MR466330} and \cite{MR387425}.
 We assume it to be true for $k \leq \ell-1$ and we prove for $k = \ell$. We partition the self-dual representations of $\G(\cO_k)$ into four components, namely split semisimple ($\Sigma^{\ss}$), split non-semisimple ($\Sigma^{\sns}$), cuspidal             ($\Sigma^{\cus}$) and non-regular ( $\Sigma^{\nreg}$) representations. 
 As mentioned earlier, the non-regular self-dual representations of $\G(\cO_\ell)$ are the same as the self dual representations of $\G(\cO_{\ell-1}).$ Hence by induction, 
 \[
 \sum_{\phi \in \Sigma^\nreg} \dim(\phi) = (q+1)q^{2\ell-3}+2.
 \]
 The dimension of any $\ss$, $\sns$, and $\cus$ representations of 
$\GL_2(\cO_\ell)$ are
$(q+1)q^{\ell-1}$, $(q^2-1)q^{\ell-2}$, and $(q-1)q^{\ell-1}$ respectively. By \autoref{thm:regular-self-dual}, we have 
$|\Sigma^\ss| = \frac{1}{2}(q^2-1)q^{\ell-2}$, $ |\Sigma^\sns| = 2q^{\ell-1}$ and $|\Sigma^\cus| = \frac{1}{2}(q-1)^2q^{\ell-2}$. Then it is easy to see that the sum of dimensions of the all self dual representations of $\GL_2(\co_\ell)$ is $(q + 1)q^{2\ell-1}+2$, and that is same as the number of involutions of $\GL_2(\cO_\ell)$ by \autoref{no. of involutions}. This proves (1). 

For (2), let $a_\ell$ be the sum of the degree of the characters of $\GU_2(\cO_\ell)$ with Frobenius Schur indicator $1$ and $b_\ell$ be the sum of the degree of the characters of $\GU_2(\cO_\ell)$ with Frobenius Schur indicator $-1$. We have $a_1 = q^2+1$ and $b_1 = q-1$ by \cite[p.27]{MR3239291}. Using this with \autoref{lem:real-char-numbers-total} and \autoref{thm:regular-self-dual}, we obtain $a_\ell + b_\ell = (q+1)q^{2\ell-1}.$ From \autoref{eqn:schur indicator formula}, we have $a_\ell - b_\ell = (q-1)q^{2\ell-1} +2.$ Therefore $a_\ell = q^{2\ell }+1$ and $b_\ell = q^{2 \ell -1}-1.$ We note that the sum of the degrees of the non-regular symplectic characters of $\GU_2(\cO_\ell)$ is given by $b_{\ell -1} = q^{2\ell-3}-1.$ The result
now follows by observing that $b_\ell \geq 1$ for $\ell \geq 1$ and $b_\ell - b_{\ell-1} = (q^2-1)q^{2\ell-3} \geq 1$ for $\ell \geq 2$.
\end{proof}
\vspace{.3cm} 
\noindent {\bf Acknowledgments:} The third named author is grateful to Amritanshu Prasad for helpful discussions regarding this project and  acknowledges the financial support provided by SERB, India, through grant SPG/2022/001099. 
 \bibliography{refs_real}{}
\bibliographystyle{siam}
\end{document}